\documentclass[11pt,reqno]{amsart}
\usepackage{cite}
\usepackage{amssymb}
\usepackage{amsfonts}
\usepackage{mathrsfs}
\usepackage[all]{xy}

\def\stab#1{\mathrm{Stab}(#1)}
\def\Image{\mathop{\mathrm{Im}}}

\newcommand{\rmv}[1]{}

\def\Ec#1#2{\ensuremath{\mathscr E_{#1}(#2)}}
\def\sp{\ensuremath{\mathsf {SP}(M,X)}}
\def\fmon#1#2{\ensuremath{\mathsf{M}_{#1}(#2)}}
\def\st#1#2{\ensuremath{\mathrm{States}_{#1}(#2)}}

\def\omst{\omega+\ast}

\def\Pc{\ensuremath{\mathrel{\mathscr{P}}}}

\def\C{\mathscr C}
\def\Om{\ensuremath{\Omega}}
\def\MM{\ensuremath{{\mathscr M}}}
\def\pv#1{\ensuremath{{\bf#1}}}

\def\ilim{\varprojlim}

\def\wh{\widehat}
\def\inv{^{-1}}
\def\p{\varphi}

\def\pinv{{\p \inv}}
\def\J{\mathrel{{\mathscr J}}} % J - relation
\def\D{{\mathcal D}} % D - relation
\def\C{{\mathscr C}}
\def\R{\mathrel{{\mathscr R}}} % R - relation
\def\H{\mathrel{{\mathscr H}}} % H - relation

\def\e<{\leq _{E}}

\def\ov#1{\ensuremath{\overline {#1}}}
\def\KG#1{\ensuremath{\mathsf K_{\pv G} (#1)}}

\def\til#1{\ensuremath{\widetilde {#1}}}

\def\looop#1{\mathop{\xymatrix{\ar@(r,u)_{#1}}}}
\def\plstar#1#2{\mathop{\xymatrix{\ar[r]^{#1} &\ar@(r,u)_{#2}}}}

\def\malce{\protect\mathbin{\hbox{\protect$\bigcirc$\rlap{\kern-8.75pt\raise0,50pt\hbox{\protect$\mathtt{m}$}}}}}
\def\Ac{\ensuremath{{\mathscr A}}}

\def\1sk{^{(1)}}

\def\to{\rightarrow}

\def\Fc{\ensuremath{\mathcal F_n}}

\def\data{\ifcase\month\or January\or February \or March\or April\or May
\or June\or July\or August\or September\or October\or November \or
December\fi\space\number\day, \number\year}
\def\flow#1{\mathrel {#1}}
\def\FFF#1{\ensuremath{#1\mathcal F_n}}
\def\dom{\mathop{\mathrm{dom}}}
\def\ran{\mathop{\mathrm{ran}}}
\def\fix{\mathop{\mathrm{fix}}}
\def\Thmname{Theorem}
\def\Propname{Proposition}
\def\Lemmaname{Lemma}
\def\Definitionname{Definition}
\newtheorem{Thm}{\Thmname}[section]
\newtheorem{Prop}[Thm]{\Propname}
\newtheorem{Lemma}[Thm]{\Lemmaname}
{\theoremstyle{definition}
\newtheorem{Def}[Thm]{\Definitionname}}
{\theoremstyle{remark}
\newtheorem{Rmk}[Thm]{Remark}}
\newtheorem{Cor}[Thm]{Corollary}
{\theoremstyle{remark}
\newtheorem{Example}[Thm]{Example}}

\newtheorem{Conjecture}[Thm]{Conjecture}

{\theoremstyle{definition}
}
{\theoremstyle{remark}
}

{\theoremstyle{remark}
}
{\theoremstyle{remark}
}
\newtheorem*{Thm*}{Theorem}

\numberwithin{equation}{section}

\title[Lower Bound for Complexity]{An Effective Lower Bound for Group Complexity of Finite Semigroups and Automata}

\author{Karsten Henckell\and John Rhodes\and Benjamin Steinberg}
\address{Department of Mathematics/Computer Science \\ New College of Florida
5800 Bay Shore Road Sarasota, Florida 34243-2109\\ \and
Department of Mathematics\\
University of California at Berkeley \\
Berkeley \\ CA 94720\\
USA\\ \and School of Mathematics and Statistics\\
Carleton University \\
1125 Colonel By Drive\\
Ottawa, Ontario  K1S 5B6 \\
Canada}
\email{KHenckell@ncf.edu\and jrhodes@math.berkeley.edu \newline\and  bsteinbg@math.carleton.ca}
\thanks{The second two authors gratefully acknowledge the support of NSERC and the Committee on Research of the Academic Senate of the University of California at Berkeley for their generous support.}

\date{December 3, 2008}

\makeindex

\begin{document}

\begin{abstract}  The question of computing the group complexity of
  finite semigroups and automata was first posed in K.~Krohn and
  J.~Rhodes, \textit{Complexity of finite semigroups}, Annals of
  Mathematics (2) \textbf{88} (1968), 128--160, motivated by the Prime
  Decomposition Theorem of K. Krohn and J.~Rhodes, \textit{Algebraic theory
    of machines, {I}:
{P}rime decomposition theorem for finite semigroups and machines},
Transactions of the American Mathematical Society \textbf{116} (1965),
450--464.  Here we provide an
  effective lower bound for group complexity.
\end{abstract}

\maketitle
\tableofcontents
\section{Introduction}
One of the oldest open problems in finite semigroup and automata theory is the computability of group complexity~\cite{KRannals}.  Krohn and Rhodes proved the Prime Decomposition Theorem~\cite{PDT} stating that every finite semigroup divides an iterated wreath product of its simple group divisors and a certain $3$-element idempotent semigroup called the ``flip-flop''~\cite[page 224]{qtheor}.  Recall that a semigroup is called aperiodic if all its subgroups are trivial.  It follows from the Prime Decomposition Theorem that if $S$ is a finite semigroup, then there is a division of the form
\begin{equation}\label{definecomplexity}
S\prec A_n\wr G_n\wr A_{n-1}\cdots \wr A_1\wr G_1\wr A_0
\end{equation}
where the $A_i$ are aperiodic semigroups and the $G_i$ are groups (where we omit the bracketing).  The group complexity (or simply complexity) of $S$ is the minimum possible value of $n$ over all wreath product decompositions \eqref{definecomplexity}.  A large part of finite semigroup theory has been developed around resolving this one problem of finding an algorithm to compute complexity.  For instance, Tilson's influential derived category construction~\cite{Tilson} was introduced~\cite{TilsonFLC,TilsonXII} exactly to provide an accessible proof to the second author's Fundamental Lemma of Complexity~\cite{flc,RFLC1,RFLC2}, stating that complexity does not drop under aperiodic surmorphisms. Ash's celebrated solution to the second author's Type II conjecture~\cite{Ash,HMPR}, and its group theoretic reformulation~\cite{RZ}, grew out of an attempt to compute lower bounds for complexity~\cite{lowerbounds1,lowerbounds2}.

Despite years of sustained work, there are not many classes of semigroups for which complexity is known to be decidable.
In~\cite{KRannals}, Krohn and Rhodes proved that complexity is decidable for completely regular semigroups.  Tilson established that complexity is decidable for semigroups with at most two non-zero $\J$-classes~\cite{2J}.  Rhodes and Tilson extended the results of~\cite{KRannals} to semigroups in the Malcev product $\pv {LG}\malce \pv A$ of local groups with aperiodic semigroups~\cite{folleyT,folleyR}.  Computable upper and lower bounds for complexity exist~\cite{lowerbounds2,lowerbounds1,Tilsonnumber,TilsonMargolis,KernelSystems,Almeida&Steinberg:2000}, but all existing bounds in the literature are known not to be tight. It is also known that the complexity pseudovarieties (above level $0$) are not finitely based, that is, admit no finite basis of pseudoidentities~\cite{finitebase}. For a modern comprehensive survey on group complexity, consult~\cite[Chapter 4]{qtheor}; also Tilson's chapters of Eilenberg~\cite{Eilenberg,TilsonXI,TilsonFLC} contain a wealth of information on complexity.  Some other sources concerning complexity include~\cite{pl,TilsonSurvey,flc,axioms,Arbib,localcomplex,RFLC1,RFLC2,TilsonFLC,complex1/2,TypeIIfalls,KernelSystems}.

The aim of this paper is to present a new lower bound for complexity that improves on all existing bounds in the literature.  The authors have some reason to believe that these bounds may be tight; only future work will tell.  This research had its origins in earlier unpublished work of the second author~\cite{flows}.

The paper is roughly organized as follows.  First we introduce the notion of flow lattices.  Then we specialize to the set-partition flow lattice associated to a group mapping monoid.  Afterwards, we reformulate the Presentation Lemma~\cite{pl,aperiodic,qtheor} in the language of flows.  We then proceed to define our lower bound.  Roughly speaking, the idea is that we are searching for certain sets and partitions that arise under all flows on automata of complexity $n$.  We begin with elementary examples of such sets and partitions and then apply closure operators that create bigger such sets and partitions.  Our lower bound consists of basically all the sets and partitions we can effectively construct in this way.

The reader is referred to~\cite{qtheor} for basic notation and definitions from finite semigroup theory; see also~\cite{Arbib,Eilenberg,Almeida:book}.

\section{Flows and Lattices}
The approach of Rhodes and Tilson to regular Type II elements of an arbitrary finite semigroup~\cite{lowerbounds2}  and of Henckell~\cite{Henckell} (see also~\cite{ouraperioidcpointlikes,qtheor}) to aperiodic pointlikes
shows that calculating lower bounds for such things amounts to
studying closure operators on certain lattices.  In the first
case, one considers the partition lattice on a regular $\R$-class of a semigroup~\cite{Redux}; in the latter
one considers the power set of a semigroup. All of these lattices
are examples of what we shall call flow lattices. The intuition
is that one builds lower bounds up \emph{from below} resulting in
a closure operator (which can also be described by intersecting
closed subsets from above). So, for instance, the lower bound for
the Type II semigroup is the smallest subsemigroup containing the
idempotents and closed under weak conjugation.  The semigroup of
aperiodic pointlikes of a semigroup $S$ is the smallest
subsemigroup of the power semigroup $P(S)$ containing the
singletons that is closed under unioning cyclic groups.  We begin by setting up our abstract formalism for flows before venturing into the lattice of interest for us.

\subsection{Lattices and closure operators}
A \emph{lattice} $L$  is a partially ordered set such that each
finite subset has a meet and a join.  In particular, by
considering empty meets and joins, $L$ has a top $T$ and a bottom
$B$.   If, in addition, $L$ has arbitrary meets and joins, then it is
called a \emph{complete lattice}.   In a complete lattice, the meet
determines the join and vice versa in the usual way.  See~\cite{qtheor} for more on lattices in the context of semigroup
theory.    In this paper, we shall primarily be interested in
finite lattices.   Any finite lattice is complete and a finite
partially ordered set is a lattice if and only if it has a top and
admits pairwise meets.

\begin{Def}[Closure operator]A \emph{closure operator} on a complete lattice $L$ is
a function $c\colon L\to L$ that is order-preserving, idempotent and
increasing.  That is, for all $\ell_1,\ell_2\in L$:
\begin{enumerate}
\item  (order-preserving)\ $\ell_1\leq \ell_2\implies \ell_1c\leq \ell_2c$;
\item (idempotent)\ $\ell_1c^2 = \ell_1c$;
\item (increasing)\ $\ell_1\leq \ell_1c$.
\end{enumerate}\index{closure operator}
\end{Def}

We use $\mathscr C(L)$ to denote the set of closure operators on
$L$.

\begin{Prop}
Let $L$ be a complete lattice.  Then the set $\mathscr C(L)$ is a complete lattice with the pointwise ordering $c_1\leq c_2$ if and only if $\ell c_1\leq \ell
c_2$ for all $\ell\in L$. The meet is pointwise, i.e., is given by
\[(\ell)\bigwedge _{\alpha\in A} c_{\alpha} = \bigwedge _{\alpha\in A} \ell
c_{\alpha},\] for $\ell\in L$, $\{c_{\alpha}\mid \alpha\in A\}\subseteq \mathscr C(L)$. The
top of $\mathscr C(L)$ is the constant function taking the value $T$
(the top of $L$); the bottom is the identity map $1_L$.
\end{Prop}
\begin{proof}
Let $\{c_{\alpha}\mid \alpha\in A\}\subseteq \mathscr C(L)$ and denote by $c$ the pointwise meet of this set.
First of all, observe that if $\ell\leq \ell'$ and $\beta\in A$, then \[\bigwedge_{\alpha\in A} \ell c_{\alpha}\leq \ell c_{\beta}\leq \ell'c_{\beta}.\]  Thus $\bigwedge_{\alpha\in A} \ell c_{\alpha}\leq \bigwedge_{\alpha} \ell'c_{\alpha}$ and hence $c$ is order-preserving.  Next we show that $c$ is increasing.  Indeed, if $\ell\in L$, then $\ell\leq \ell c_{\alpha}$, for all $\alpha\in A$ and hence $\ell\leq (\ell)\bigwedge_{\alpha\in A}c_{\alpha}= \ell c$.

Since $c$ is increasing, $\ell c\leq \ell c^2$.  Thus we need only establish the reverse inequality.  Now if $\alpha\in A$, then since $c$ is pointwise below $c_{\alpha}$ we have \mbox{$\ell c^2\leq \ell c_{\alpha}^2=\ell c_{\alpha}$} and hence $\ell c^2\leq \bigwedge_{\alpha\in A} \ell c_{\alpha} = \ell c$.  This concludes the proof that $\mathscr C(L)$ is closed under pointwise meets.  Hence $\mathscr C(L)$ is a lattice and the constant map to $T$ is the top of $\mathscr C(L)$.  Since closure operators are increasing, plainly $1_L$ is the bottom.
\end{proof}

We remark that the join in $\mathscr C(L)$ is
the determined join~\cite[Chapter 6]{qtheor} and is not in general the pointwise join. (The
determined join of a subset is the meet of all its upper bounds.)

The easiest way to understand these notions is via the following
alternative characterization of a closure operator.  If $c$ is a
closure operator on $L$, an element $\ell\in L$ is called
\emph{stable} or \emph{closed} if $\ell c=\ell$.  It is well known
that the set of stable elements $Lc$ is a meet-closed subset of $L$~\cite[Proposition 6.3.6]{qtheor}.  Conversely, if $K\subseteq L$ is a meet-closed
subset (and so $T\in K$), then the function $c_K\colon L\to L$ given by
\begin{equation}\label{closureformula}
\ell c_K = \bigwedge \{k\in K\mid \ell\leq k\}
\end{equation}
is a closure operator with $K=Lc_K$.  Moreover, if $c\in \mathscr
C (L)$, then $c_{Lc} =c$~\cite[Proposition 6.3.6]{qtheor}.  Hence from \eqref{closureformula} it is
immediate that
\begin{equation}\label{reversal} c_1\leq c_2\ \text{in}\ \mathscr
C(L) \iff \mathop{\mathrm{Im}} c_1\supseteq \mathop{\mathrm{Im}}
c_2.\end{equation} We remark that the reversal in \eqref{reversal}
is crucial.  The above discussion shows that if $\mathsf C(L)$
denotes the collection of meet-closed subsets of $L$ ordered by
\emph{reverse} inclusion, then $\mathsf C(L)$ is a complete
lattice with join given by intersection. The
bottom is the set $L$, the top is the set $\{T\}$.  The meet is
determined, namely the meet of a subset $W$ is the intersection of
all meet-closed subsets containing $W$.  Equivalently, one takes
the union of $W$ and then closes it under meets.   Our discussion
establishes the following well-known proposition.

\begin{Prop}\label{closure}
Let $L$ be a complete lattice.  Then the complete lattices $(\mathscr C(L),\leq)$ and \mbox{$(\mathsf C(L), \supseteq)$} are isomorphic.
\end{Prop}

Henceforth, we identify $\mathscr C(L)$ and $\mathsf C(L)$
and so we drop the notation $\mathsf C(L)$.

\subsection{Abstract flows}
Let $L$ be a complete lattice.  Then $L^2=L\times L$ is a complete
lattice under coordinate-wise ordering.  The meet and join are
coordinate-wise.

A \emph{binary relation}\index{binary relation} on a set $A$ is a subset of $A\times A$.  If $f$ is a binary relation, we write $a\flow f b$ to indicate $(a,b)\in f$.  Binary relations form a monoid $B(A)$ where composition is given by $x \mathrel{fg} y$ if and only if there exists $z\in A$ so that $x\mathrel{f}z\mathrel{g} y$ for $x,y\in A$ and $f,g\in B(A)$.  The identity $I_A$ is just the diagonal $\{(a,a)\mid a\in A\}$.   Sometimes, it is convenient to identify $f\in B(A)$ with the map $f'\colon A\to P(A)$ (the power set of $A$) given by $af'= \{b\in A\mid a\mathrel f b\}$.  In particular, we will abuse notation and denote the function and the relation by the same letter. Consequently, any partial function $f\colon A\to A$ can be viewed as a binary relation.  The associated subset of $A\times A$ is the graph of $f$, i.e., the set $\{(a,f(a))\mid a\in \dom f\}$.  For instance, the identity of $B(A)$ is the binary relation corresponding to the identity function on $A$.  Given any subset $A'\subseteq A$, one can consider the partial identity $1_{A'}$.  The corresponding binary relation is $\{(a,a)\mid a\in A'\}$.
The case of interest for us will be binary relations on a complete lattice, but we will only be interested in those relations that preserve the lattice structure.

\begin{Def}[Abstract flow]An \emph{abstract $L$-flow} is an element of $\mathscr
C(L^2)$.  If $L$ is understood from the context, then we simply
call an element of $\mathscr C (L^2)$ an \emph{abstract flow}.\index{abstract flow}
\end{Def}

If $f$ is an abstract flow, then $\mathop{\mathrm{Im}} f$ is a meet-closed
subset of $L^2$ and hence a binary relation on $L$ that we denote $\ov f$.  So to make clear our notational conventions: $f\in \mathscr C(L^2)$ denotes a closure operator on $L\times L$ and $\ov f\in B(L)$ stands for the corresponding binary relation $\mathop{\mathrm{Im}} f$.  This leads
us to the suggestive notation $\ell_1\mathrel f \ell_2$ for the
value of $f$ on $(\ell_1,\ell_2)\in L^2$. Since many of our flows
come from automata, we also use the illustrative notation
\[\ell_1\xrightarrow{f} \ell_2.\] The elements of $\ov f$ are
called the \emph{stable pairs}\index{stable pair} of $f$.  We write $\ell_1\flow
{\ov f} \ell_2$ to indicate $(\ell_1,\ell_2)$ is stable for $f$; we
also use the pictorial notation \[\ell_1\xrightarrow{\ov f}\ell_2\] to
indicate $(\ell_1,\ell_2)$ is stable under $f$.

We now want to show that $\mathscr C(L^2)$ is a submonoid of  $B(L)$  under the identification $f\mapsto \ov f$.  More precisely, we identify $\mathscr C(L^2)$ with the set of those binary relations on $L$ that are meet-closed (as subsets of $L\times L$).

\begin{Prop}\label{flowmonoidismonoid}
$\mathscr C(L^2)$ is a submonoid of $B(L)$ under the
identification $f\mapsto \ov f$.  That is, a multiplication can be
defined on $\mathscr C(L^2)$ by setting $fg$ to be the closure
operator associated to $\ov f\ov g$, turning $\mathscr C(L^2)$ into a monoid.
\end{Prop}
\begin{proof}
It is clear that the set of stable pairs of the identity relation $I_L$ is
meet-closed and so belongs to $\mathscr C(L^2)$ under our identification. We need to show
that if $f,g\in \mathscr C(L^2)$ then the set of stable pairs of
$\ov {fg}= \ov f\ov g$ is meet-closed.  The empty meet is $(T,T)$
where $T$ is the top of $L$.  Since the sets of stable pairs of $f$
and $g$ are meet-closed, \[T\xrightarrow{\ov f}T\xrightarrow {\ov g} T\] and so
$T\xrightarrow {\ov{fg}} T$. Suppose $\{(\ell_{\alpha},\ell'_{\alpha})\}_{\alpha\in A}$
is a non-empty subset of the set of stable pairs of $\ov {fg}$.
Then, for each $\alpha\in A$, there is an element $\ell''_{\alpha}\in L$
so that \[\ell_{\alpha}\xrightarrow {\ov f} \ell''_{\alpha}\xrightarrow {\ov
g} \ell'_{\alpha}.\]  Since the stable pairs of $f$ and $g$ are meet-closed, we then obtain \[\bigwedge _{\alpha\in A} \ell_{\alpha}\xrightarrow
{\ov f} \bigwedge_{\alpha\in A}\ell''_{\alpha} \xrightarrow {\ov g} \bigwedge
_{\alpha\in A}\ell'_{\alpha}\] showing that \[\bigwedge _{\alpha\in A}
\ell_{\alpha}\xrightarrow {\ov{fg}} \bigwedge_{\alpha\in A} \ell'_{\alpha}.\]
Since meets in $L^2$ are coordinate-wise, this completes the proof.
\end{proof}

We remark that $I_L$ has stable pairs of the form $\{(\ell,\ell)\mid
\ell\in L\}$. Hence  \[\ell_1 \flow{I_L} \ell_2 = (\ell_1\vee
\ell_2,\ell_1\vee \ell_2).\] So the multiplicative identity $I_L$ of
$\mathscr C(L^2)$ is not the identity closure operator (which has
stable set $L^2$ and is the bottom of $\mathscr C(L^2)$).

We shall call $\mathscr C(L^2)$ the \emph{abstract flow monoid} on
$L$. From its description as a submonoid of $B(L)$ we immediately
obtain:

\begin{Prop}\label{isanorderedmonoid}
$\mathscr C(L^2)$ is an ordered monoid.  That is, $f_1\leq f_2$ and
$g_1\leq g_2$ implies $f_1g_1\leq f_2g_2$.
\end{Prop}

\begin{Rmk}
It is convenient to know which binary relations come from $2$-variable closure operators from the point of view of relations on $L$ as maps $L\to P(L)$.  We observe that $f\subseteq L\times L$ is meet-closed if and only if $\ell'_{\alpha}\in \ell_{\alpha}f$, for $\alpha\in A$, implies $\bigwedge_{\alpha\in A}\ell'_{\alpha}\in \left(\bigwedge_{\alpha\in A} \ell_{\alpha}\right)f$ where view $f$ as a map $L\to P(L)$ for the moment.  This is analogous to relational morphisms.
\end{Rmk}

\subsubsection{From one-variable to two-variable closure operators and back again}
It turns out to be useful to identify $\mathscr C(L)$ with a certain submonoid of the abstract flow monoid $\mathscr C(L^2)$.  This will allow one-variable operators to act on the left and right of abstract flows by inner translations.

\begin{Prop}\label{twoasone}
Let $f\in \mathscr C(L)$.  Define $\p\colon \mathscr C(L)\to B(L)$ by $f\mapsto 1_{\Image f}$, where the latter partial identity is viewed as a binary relation.  Then:
\begin{enumerate}
\item $f\p\in \mathscr C(L^2)$;
\item $\p$ is an order-embedding;
\item The identity closure operator maps to the identity relation;
\item $f\p g\p = (f\vee g)\p$;
\item $\mathscr C(L)\p$  is an idempotent, commutative submonoid of $\mathscr C(L^2)$ isomorphic to $\mathscr C(L)$ with the join operation.
\end{enumerate}
\end{Prop}
\begin{proof}
The first three items are trivial.  For (4), $1_{\Image f}1_{\Image g} = 1_{\Image f\cap \Image g} = 1_{\Image f\vee g}$.   The final item is immediate from the previous one.
\end{proof}

From now on we identify $f\in \mathscr C(L)$ with $f\p$ and drop the latter notation. To make things more concrete, if $f\in \mathscr C(L)$ is a closure operator on $L$, then the corresponding meet-closed binary relation is \[\ov f=\{(\ell,\ell)\mid \ell f=\ell\}.\]
Proposition~\ref{twoasone} allows us to view the join-lattice $\mathscr C(L)$ as operating on the left and right of $\mathscr C(L^2)$.   The following proposition is immediate from the definitions and is stated merely for the convenience of the reader.

\begin{Prop}
Let $h\in \mathscr C(L)$ and $f\in \mathscr C(L^2)$.  Then $\ell\xrightarrow{\ov {hf}}\ell'$ if and only if $\ell h=\ell$ and $\ell\xrightarrow{\ov f}\ell'$.  Dually, $\ell\xrightarrow{\ov {fh}}\ell'$ if and only if $\ell' h=\ell'$ and $\ell\xrightarrow{\ov f}\ell'$.
\end{Prop}

Viewing flows as binary relations, if $h\in \mathscr C(L)$ and $f\in \mathscr C(L^2)$, then $\ov{hf}$ is the binary relation obtained by restricting the domain of $\ov f$ to $\Image h$.  Similarly, $\ov{fh}$ is the binary relation obtained by restricting the range of $\ov f$ to $\Image h$.

The following proposition establishes the basic properties of our
action.

\begin{Prop}\label{actionfacts}
Let $g,g_1,g_2\in \mathscr C(L)$ and $f,f'\in \mathscr C(L^2)$.
\begin{enumerate}
\item $g_1g_2f=g_2g_1f = (g_1\vee g_2)f$;
\item $fg_1g_2=fg_2g_1=f(g_1\vee g_2)$;
\item $1_Lg=g=g1_L$;
\item $f\leq g_1fg_2$.
\end{enumerate}
\end{Prop}
\begin{proof}
Proposition~\ref{twoasone} immediately yields (1), (2) and (3).  Item (4) follows from $1_L$ being the bottom of $\mathscr C(L)$ and Proposition~\ref{twoasone}.
\end{proof}

\subsubsection{Back-flow, forward-flow and star}
To any binary relation $f$ on $L$, we can associate three subsets of $L$:  the domain $\dom f$, the range $\ran f$ and the fixed point set $\fix f$ (where $\fix f=\{\ell\in L\mid (\ell,\ell)\in f\}$).  It is immediate that if $\ov f$ is the relation associated to $f\in \mathscr C(L^2)$ (and so $\ov f$ is meet-closed), then these sets are meet-closed and hence define one-variable closure operators (which we can view as two-variable closure operators in our usual way).

\begin{Prop}
Let $f\in \mathscr C(L^2)$.  Then $\dom \ov f$, $\ran \ov f$ and $\fix \ov f$ are meet-closed subsets of $L$ and hence correspond to one-variable closure operators.
\end{Prop}
\begin{proof}
Clearly, $T\in \dom \ov f$.
If $\{\ell_{\alpha}\mid \alpha\in A\}\subseteq \dom \ov f$, then we can find, for all $\alpha\in A$, an element $\ell'_{\alpha}\in L$ so that $\ell_{\alpha}\xrightarrow{\ov f}\ell'_{\alpha}$.  Then since $f$ is a closure operator on $L^2$, we have $\bigwedge_{\alpha\in A} \ell_{\alpha}\xrightarrow{\ov f}\bigwedge_{\alpha\in A} \ell'_{\alpha}$ and so $\bigwedge_{\alpha\in A}\ell_{\alpha}\in \dom \ov f$.  Similarly, $\ran \ov f$ is meet-closed.

For the fixed point set, suppose $\ell_{\alpha}\xrightarrow{\ov f}\ell_{\alpha}$ all $\alpha\in A$.  Then  \[\bigwedge_{\alpha\in A} \ell_{\alpha}\xrightarrow{\ov f}\bigwedge_{\alpha\in A} \ell_{\alpha},\] whence $\bigwedge_{\alpha\in A} \ell_{\alpha}\in \fix \ov f$, completing the proof.
\end{proof}

Because we deal with deterministic automata (but not necessarily co-deterministic automata), we have little occasion to use $\ran \ov f$.  Let us give names to the associated closure operators for the other two sets.
\begin{Def}[Back-flow and Kleene star]
Let $f\in \mathscr C(L^2)$.  Then we define the following one-variable closure operators associated to $f$:
\begin{enumerate}
\item Define $\overleftarrow{f}$ to be the closure operator on $L$ with image $\dom \ov f$.  It is called \emph{back-flow} along $f$.
\item Define $f^*$ to be the closure operator on $L$ with image $\fix \ov f$.  It is called the \emph{Kleene star} of $f$.
\end{enumerate}\index{back-flow}\index{Kleene star}
\end{Def}

The names will be motivated a little bit later when we look at the set-partition lattice and flows on automata.  For instance, Remark~\ref{explainbackflow} will motivate back-flow.

Since taking the domain and fixed-point sets are order-preserving it follows that the maps $f\mapsto \overleftarrow{f}$ and $f\mapsto f^*$ are order-preserving.  Also since $\fix \ov f\subseteq \dom \ov f$, we have $\overleftarrow f\leq f^*$.

It is sometimes convenient to work with a direct description of the back-flow closure operator. Let
$\pi_b,\pi_f\colon L^2\to L$ be the projections
\begin{align*}
(\ell,\ell')\pi_b &= \ell\\
(\ell,\ell')\pi_f &= \ell'.
\end{align*}
The letters $b$ and $f$ stand for back and front (we are thinking
in pictures $\ell\rightarrow \ell'$).  Notice that $\pi_f,\pi_b$
are complete lattice homomorphisms.

\begin{Prop}\label{stableforcheck}
Let $f\in \C(L^2)$ and $\ell\in L$.  Then
\begin{equation}\label{backflowformula}
\ell\overleftarrow{f} = \left(\ell\xrightarrow{f} B\right)\pi_b
\end{equation}
where, as usual, $B$ is the bottom of $L$.
\end{Prop}
\begin{proof}
By definition $\ell\overleftarrow{f}$ is the least element $\ell'$ in the domain of $\ov f$ such $\ell\leq \ell'$. So suppose $(\ell',\ell'')$ is a stable pair for $f$ with $\ell\leq \ell'$.  Then $(\ell,B)\leq (\ell',\ell'')$ and so $(\ell\xrightarrow{f} B)\leq (\ell',\ell'')$.  Thus the right hand side of \eqref{backflowformula} is the minimal element of $\dom \ov f$ that is above $\ell$.
\end{proof}

This proposition explains to some extent the terminology back-flow. Since $f$ is
order-preserving, $\ell\overleftarrow{f}\leq (\ell\flow f
\ell')\pi_b$ for any $\ell,\ell'\in L$. So $\ell\overleftarrow{f}$
is picking up whatever always flows back to $\ell$.  We are also interested in what must flow forward.

\begin{Def}[Forward-flow]
If $f\in \mathscr C(L^2)$ and $\ell\in L$, we define an order-preserving map $\overrightarrow{f}\colon L\to L$ by \[\ell\overrightarrow{f} = (\ell\xrightarrow{f} B)\pi_f,\] where $B$ is the
bottom of $L$.  We call $\overrightarrow{f}$ \emph{forward-flow}
along $f$.\index{forward-flow}
\end{Def}

That $\overrightarrow f$ is order-preserving follows as $\ell\leq \ell'$ implies $(\ell,B)\leq (\ell',B)$.
Notice that we use $\overleftarrow f$ for back-flow
since the arrow points backwards and $\overrightarrow f$ for
forward-flow for the opposite reason.  See Remark~\ref{explainbackflow} below for motivation.

\begin{Rmk}\label{oneastwonoflow}
If we view $f\in \mathscr C(L)$ as a two-variable closure operator, then $\dom \ov f = \Image f$ and so $f=\overleftarrow{f}$.  In fact, $\ell\xrightarrow{f} B = (\ell f,\ell f)$ and so $f=\overrightarrow{f}$, as well.
\end{Rmk}

\begin{Prop}\label{multofforwardflow}
The map $f\mapsto \overrightarrow f$ from $\mathscr C(L^2)$ to the
monoid of order-preserving maps on $L$ satisfies, for  $f,g\in
\mathscr C(L^2)$,
\begin{equation}\label{noaction}
\overrightarrow f \overrightarrow g\leq \overrightarrow{fg}
\end{equation}
where the order is taken pointwise.
It also sends the identity to the
identity.
\end{Prop}
\begin{proof}
Since $\dom I_L = L$, which in turn is the image of the identity closure operator on $L$, it follows $\overleftarrow{I_L}$ is the identity map on $L$.

 Suppose $f,g\in \mathscr C(L^2)$ and
$\ell\in L$. If $\ell\xrightarrow {fg} B = (\ell_1,\ell_2)$, then there
exists $\ell'\in L$ such that
$\ell_1\xrightarrow{\ov f}\ell'\xrightarrow {\ov g}\ell_2.$  Since $\ell\leq \ell_1$ and
$B\leq \ell'$, it follows that $\ell\xrightarrow f B\leq (\ell_1,\ell')$
and so $\ell\overrightarrow f\leq \ell'$. Because $B\leq \ell_2$, we
have $\ell'\overrightarrow g\leq \ell_2$. Thus \[\ell\overrightarrow
f \overrightarrow g\leq \ell'\overrightarrow g\leq \ell_2 =
\ell\overrightarrow{fg},\] as required.
\end{proof}

In general equality does not hold in \eqref{noaction}.  However,
in the situation that will be of primary interest to us, it will
turn out to hold.  Namely, when there is no back-flow, equality
holds as the following proposition demonstrates.

\begin{Prop}\label{nobackflow}
Suppose $\ell\in L$ and $f,g\in \mathscr C(L^2)$ are such that
$\ell\overleftarrow{f} =\ell$,
$(\ell\overrightarrow{f})\overleftarrow{g} = \ell\overrightarrow{f}$.
Then
\[\ell\overrightarrow f \overrightarrow g = \ell\overrightarrow{fg}.\]
\end{Prop}
\begin{proof}
By Proposition~\ref{multofforwardflow}, it suffices to prove that
$\ell\overrightarrow{fg}\leq \ell\overrightarrow f\overrightarrow
g$.  Let $\ell' = \ell\overrightarrow f$. By hypothesis on $\ell$
and $f$, $\ell\xrightarrow f B = (\ell,\ell')$. By hypothesis on $\ell'$
and $g$, $\ell'\xrightarrow g B = (\ell',\ell'\overrightarrow g)$. Thus
$\ell \xrightarrow{\ov {fg}} \ell'\overrightarrow g$. Hence
\[\ell\xrightarrow {fg} B\leq \ell\xrightarrow {fg} \ell'\overrightarrow g = (\ell,\ell'\overrightarrow g),\]
establishing $\ell\overrightarrow{fg}\leq \ell\overrightarrow
f \overrightarrow g$, as required.
\end{proof}

\subsection{Flow lattices}
For this section, fix a finite non-empty alphabet $X$.  We shall
need the notion of  an $X$-flow lattice.   Examples of flow
lattices arise from trying to compute complexity via the
Presentation Lemma~\cite{pl,aperiodic,qtheor}, as well as when trying to
compute pointlikes for certain pseudovarieties~\cite{Henckell,ouraperioidcpointlikes}. If $X$ is a set,
$X^*$ denotes the free monoid generated by $X$.

\begin{Def}[Flow lattice]
An \emph{$X$-flow lattice} is a complete lattice $L$ equipped with
a map $\Phi\colon X\to \mathscr C(L^2)$ or equivalently, abusing
notation, a homomorphism $\Phi\colon X^*\to \mathscr C(L^2)$. The closure operator $w\Phi$, for $w\in
X^*$, is called \emph{free flow along $w$}\index{free flow} and we denote it in
arrow notation by $\xrightarrow{w}$.\index{flow lattice}
\end{Def}

Let us give a motivating example.  Fix for the rest of the paper an $X$-generated
finite group mapping monoid $M$~\cite{Arbib,qtheor}.  That is, $M$ has a
$0$-minimal regular ideal $I$ (necessarily unique), containing a
non-trivial group, such that $M$ acts faithfully on both the left
and right of $I$.  Fix also an $\R$-class $R$ of $I$, which shall be termed the \emph{distinguished} $\R$-class of $M$.  We view
$(R,M)$ as a faithful partial transformation monoid~\cite[Chapter 4]{qtheor}.

\begin{Def}[Set flow lattice]
Take $\mathsf S(M,X)=(P(R),\subseteq)$.
 This is called the \emph{set flow
lattice} for $M$. To make $\mathsf S(M,X)$ into an $X$-flow lattice,
define $x\Phi$ by $X\xrightarrow{\ov x} Y$ if and only if $Xx\subseteq
Y$.\index{set flow lattice}\index{flow lattice!set}
\end{Def}
  In this paper, we use the convention that if $(Q,M)$ is a
partial transformation monoid or automaton and $qm$ is not
defined, then we write $qm=\emptyset$.

\begin{Prop}
$\mathsf S(M,X)$ is an $X$-flow lattice.
\end{Prop}
\begin{proof}
The top of $\mathsf S(M,X)$ is $R$ and clearly $Rx\subseteq R$ so
the set of stable pairs of $x\Phi$ is closed under empty meets. If
$Ux\subseteq Y$ and $Zx\subseteq W$, then clearly $(U\cap Z)x\subseteq Y\cap W$.  So
$\ov{x\Phi}$ is closed under finite meets.
\end{proof}

We remark that the set flow lattice does not depend on the choice
of $R$ since all $\R$-classes are isomorphic via left
multiplication. Another important example is the set-partition
flow lattice of $M$. In this paper, we do not distinguish between a partition and its associated equivalence relation.

\begin{Def}[Set-partition flow lattice]
The \emph{set-partition flow lattice} $\mathsf {SP}(M,X)$ consists
of all pairs $(Y,P)$ where $Y\subseteq R$ and $P$ is a partition
on $Y$.  This is a lattice where $(Y,P)\leq (Z,Q)$ if and only if
$Y\subseteq Z$ and $y\flow P y'$ implies $y\flow Q y'$; in other
words the inclusion of $Y$ into $Z$ induces a well-defined map
$Y/P\to Z/Q$.\index{set-partition flow lattice}\index{flow lattice!set-partition}
\end{Def}

 It is easily verified that $\mathsf {SP}(M,X)$ is a
lattice. The top is given by $(R,\{R\})$, that is, the set $R$ with
a single block for the partition.  The bottom is
$(\emptyset,\emptyset)$.  The meet is given by
\[(U,P)\wedge (Y,Q) = (U\cap Y, P\wedge Q)\] where the blocks of
$P\wedge Q$ consist of all non-empty intersections of the form
$B\cap B'$ with $B$ a block of $P$ and $B'$ a block of $Q$.  The
join is easily verified to be given by
\[(U,P)\vee (Y,Q) = (U\cup Y,P\vee Q)\] where $P\vee Q$ is the
transitive closure of $P\cup Q$ viewed as a relation on $U\cup Y$,
that is, the equivalence relation on $U\cup Y$ generated by $P$
and $Q$.  Again there is no dependence on the choice of the
$\R$-class $R$ in the definition of $\mathsf {SP}(M,X)$.

The set-partition flow lattice is used in computing complexity via
the Presentation Lemma~\cite{pl,aperiodic}.  More details will be
given in the next section.

To make $\mathsf {SP}(M,X)$ an $X$-flow lattice, we declare
\[(U,P)\xrightarrow {\ov x} (Y,Q)\] if and only if $Ux\subseteq Y$ and the partial
function $\cdot x\colon U\to Y$ induces a well-defined partial
\emph{injective} map $\cdot x\colon U/P\to Y/Q$.  This means that if
$m,n\in U$ and $mx,nx\in R$ (and hence in $Y$), then \[m\flow P n\iff mx\flow Q nx.\] In
this way we have defined $x\Phi$.

\begin{Prop}\label{setpartitionflowworks}
$\mathsf {SP}(M,X)$ is an $X$-flow lattice.
\end{Prop}
\begin{proof}
Clearly $(R,\{R\})\xrightarrow{\ov x}(R,\{R\})$.  Suppose now that
\begin{equation}\label{eqsetpartitionflowworks}
(U_1,P_1)\xrightarrow {\ov{x}}(Y_1,Q_1)\ \text{and}\
 (U_2,P_2)\xrightarrow {\ov{x}} (Y_2,Q_2).
\end{equation}
Then, as we saw above,
\[(U_1\cap U_2)x\subseteq Y_1\cap Y_2\] so it suffices to show
that \[(U_1\cap U_2)/(P_1\wedge P_2)\to (Y_1\cap Y_2)/(Q_1\wedge Q_2)\] is
 a partial injective function.   Let $m,n\in U_1\cap U_2$ and suppose $x\in X$ is such that $mx,nx\in R$. Then $m\flow{P_1\wedge P_2} n$ if and only if $m \flow {P_1} n$ and
 $m\flow {P_2} n$.  But by \eqref{eqsetpartitionflowworks} this
 occurs if and only if $mx\flow {Q_1} nx$ and $mx\flow {Q_2} nx$,
 that is, if and only if $mx \flow{Q_1\wedge Q_2} nx$.  This
 completes the proof that the set of stable pairs of $\xrightarrow{x}$ is
 meet-closed.
\end{proof}

Notice that there is a natural lattice homomorphism from $\mathsf
{SP}(M,X)$ to $\mathsf S(M,X)$ preserving the $X$-flow lattice
structure.  Here by a lattice homomorphism, we mean a map
preserving both meets and joins.

It is worth describing the closure operator $w\Phi$ for strings $w\in X^*$.
\begin{Prop}\label{closureoperforstrings}
Let $w\in X^*$.  Then $(Y,P)\xrightarrow{\ov w}(Z,Q)$ if and only if $Yw\subseteq Z$ and $r\flow P s\iff rw\flow Q sw$ for all $r,s\in Y$ with $rw,sw\in R$.
\end{Prop}
\begin{proof}
Let $w=x_1\cdots x_n$ with the $x_i\in X$.  First assume $(Y,P)\xrightarrow{\ov w}(Z,Q)$.  Then we can find $(Y_1,P_1),\ldots,(Y_{n-1},P_{n-1})$ so that  \[(Y,P)\xrightarrow{\ov x_1}(Y_1,P_1)\longrightarrow \cdots \longrightarrow (Y_{n-1},P_{n-1})\xrightarrow{\ov x_n}(Z,Q).\]  Then $Yx_1\subseteq Y_1, Y_1x_2\subseteq Y_2,\ldots, Y_{n-1}x_n\subseteq Z$.  Thus $Yw\subseteq Z$.  Also right multiplication by $w$ induces a partial injective map $Y/P\to Z/Q$, namely the composition of partial injective maps \[Y/P\xrightarrow{\cdot x_1}Y_1/P_1\longrightarrow\cdots\longrightarrow Y_{n-1}/P_{n-1}\xrightarrow{\cdot x_n}Z/Q.\]

Conversely, suppose the conditions of the proposition holds.  Set $Y_0=Y$, $Y_n=Z$ and $Y_{i+1}=Y_ix_{i+1}$ for $0\leq i\leq n-2$.  From $Yw\subseteq Z$, it follows $Y_{n-1}x_n\subseteq Z$.  Assume inductively that we have a partition $P_i$ on $Y_i$, for \mbox{$0\leq i\leq n-2$}, so that $(Y_{i-1},P_{i-1})\xrightarrow{\ov x_i}(Y_i,P_i)$ and  \[rx_{i+1}\cdots x_n\flow Q sx_{i+1}\cdots x_n\iff r\flow {P_i} s\] for $r,s\in Y_i$ and $rx_{i+1}\cdots x_n, sx_{i+1}\cdots x_n\in R$.  The base case is the hypothesis (take $P_0=P$).  For the general case, noting that $Y_{i+1} = Y_ix_{i+1}$, set $rx_{i+1}\flow {P_{i+1}} sx_{i+1}$ if and only if $r\flow {P_i} s$.  This is well defined because if $rx_{i+1}=r'x_{i+1}$, then $rx_{i+1}\cdots x_n=r'x_{i+1}\cdots x_n$ and hence $r\flow {P_i} r'$ by hypothesis.  It is then immediate from the construction that $(Y_i,P_i)\xrightarrow{\ov x_{i+1}}(Y_{i+1},P_{i+1})$.  Suppose $rx_{i+2}\cdots x_n,sx_{i+2}\cdots x_n\in R$ for $r,s\in Y_{i+1}$.  Write $r=r'x_{i+1},s=s'x_{i+1}$ with $r',s'\in Y_i$.  Then $r\flow{P_{i+1}} s$ if and only if $r'\flow {P_i}s'$, if and only if $r'x_{i+1}\cdots x_n\flow Q s'x_{i+1}\cdots x_n$, if and only if $rx_{i+2}\cdots x_n\flow Qsx_{i+2}\cdots x_n$.  This completes the induction.  By construction, we have  \[(Y,P)\xrightarrow{\ov x_1}(Y_1,P_1)\longrightarrow \cdots \longrightarrow (Y_{n-1},P_{n-1})\xrightarrow{\ov x_n}(Z,Q)\] and so $(Y,P)\xrightarrow{\ov w}(Z,Q)$, as required.
\end{proof}

\begin{Def}[Points]
By a \emph{point} of $\mathsf S(M,X)$ we mean an element of $R$ (viewed as a singleton).  By a point of $\mathsf {SP}(M,X)$ we main a pair $(\{r\},\{\{r\}\})$, which we denote simply by $(r,r)$.\index{points}
\end{Def}

A key property of points in either of the above two settings is that if $p$ is a point and $x\in X$, then $p\xrightarrow{x} B =(p,q)$ where  $q$ is either the bottom or a point.  More precisely, we have the following statement, which is immediate from the definitions.

\begin{Prop}\label{SPpointed}
Let $(r,r)$ be a point of $\mathsf {SP}(M,X)$ and let $x\in X$.  Then \[(r,r)\xrightarrow x  B = \begin{cases} (rx,rx) & rx\in R\\ B & \text{else.}\end{cases}\]
\end{Prop}

We are now in a position to explain the terminology forward-flow and back-flow.

\begin{Rmk}[Explanation of back-flow]\label{explainbackflow}
First consider the set flow lattice $\mathsf S(M,X)$. Then it is easy to see that $U\xrightarrow{x} B = (U,Ux)$ for $x\in X$ and $U\subseteq R$.  Thus $U\overleftarrow{x} = U$ and $U\overrightarrow{x}=Ux$.  That is, sets only flow forward. On the other hand, back-flow can occur for the set-partition flow
lattice.  For example, suppose that $m,n\in U$, $x\in X$ and $mx=nx\in R$.  Assume
further that $m$ and $n$ are in different blocks of the partition
$P$.  Then \[\left((U,P)\xrightarrow {\, x\,}(\emptyset,\emptyset)\right)\pi_b\] will be of the form $(U,P')$
where in $P'$ the blocks of $m$ and $n$ are joined together (and
maybe more).  Thus when one flows along $x$, there is some information flowing backwards.
\end{Rmk}

The following important proposition gives a better understanding of back-flow and forward-flow.

\begin{Prop}\label{stringsareasy}
Let $w\in X^*$ and suppose that $(Y,P)\overleftarrow{w}=(Y,P)$.  Let $P=\{B_1,\ldots,B_r\}$ where $B_1,\ldots,B_k$ are the blocks of $P$ with $B_iw\neq \emptyset$. Then $B_1w,\ldots, B_kw$ are disjoint and $(Y,P)\overrightarrow{w}=(Yw,\{B_1w,\ldots, B_kw\})$.
\end{Prop}
\begin{proof}
Let $(Y,P)\overrightarrow{w}=(Z,Q)$.  Then $(Y,P)\xrightarrow{w}(\emptyset,\emptyset) = ((Y,P),(Z,Q))$ by hypothesis.  By Proposition~\ref{closureoperforstrings}, it follows that $Yw\subseteq Z$ and $Y/P\xrightarrow{\cdot w}Z/Q$ is a partial injective map.  Consequently, $B_iw\cap B_jw\neq \emptyset$ implies $i=j$.  Setting $P'= \{B_1w,\ldots, B_kw\}$, we have $(Yw,P')\in \sp$ and $(Yw,P')\leq (Z,Q)$.  Moreover, $(Y,P)\xrightarrow{\ov w}(Yw,P')$ since $Y/P\xrightarrow{\cdot w}Yw/P'$ is trivially an injective partial map.  Thus $(Z,Q) = (Yw,P')$ as required.
\end{proof}

There is an straightforward  generalization of these flow lattices
to any $X$-generated faithful partial transformation monoid
$(Q,M)$.

\subsection{Flows on automata}
One can build new closure operators on $L^2$ via composition and Kleene star.  In fact, there is a convenient formalism, via automata,  to construct more elaborate closure operators.  Given an $X$-flow lattice, one has an immediate interpretation of any automaton over the alphabet $X$ as a closure operator.  This motivates our arrow notation and the use of the Kleene star.

Fix a lattice $L$.  By an $L$-automaton $\mathscr A=(Q,\delta)$ we
mean a finite directed graph with vertex set $Q$ whose edge set $E$  is
labelled by elements of $\mathscr C(L^2)$ via $\delta\colon E\to \mathscr C(L^2)$.
 We continue to
fix a finite alphabet $X$.  Suppose, in addition, $L$ is an $X$-flow
lattice and $\mathscr A=(Q,X)$ is an automaton (possibly
non-deterministic)~\cite{EilenbergA} over $X$ (we say an
$X$-automaton). Here $Q$ denotes the state set, while the transition
function is assumed to be understood.  All automata are assumed finite. By abusing the distinction
between elements $x\in X$ and the associated free flow $x\Phi$ operator we
may view $\mathscr A$ as an $L$-automaton.   If $\mathscr A$ is a
partial deterministic  automaton (that is, $X$ acts on $Q$ by partial functions~\cite{EilenbergA}), then the \emph{completion} $\Ac^{\square}$
of $\Ac$ is obtained by adding a sink state $\square$~\cite{EilenbergA}.  Often we will write ``partial automaton'' as an abbreviation for partial deterministic automaton.

 By convention, if we draw a finite graph
  with the vertices labelled by lattice elements from $L$ and the edges
labelled by various $\ov {f}$ with $f\in \mathscr C(L^2)$, then we
assume the lattice elements labelling the initial and terminal
vertices of each edge $e$ form a stable pair for the closure
operator labelling $e$.

We remark that many of our definitions make sense for any lattice
$L$.  Only when speaking about elements of $X$ do we need to
consider $X$-flow lattices.

\begin{Def}[Flow on an $L$-automaton]
Let $\mathscr A=(Q,\delta)$ be an $L$-automaton.  By an \emph{$L$-flow} on $\mathscr A$, or just a \emph{flow} if $L$ is understood, we mean a function
\mbox{$F\colon Q\to L$} satisfying $qF\xrightarrow{\ov {e\delta}} q'F$ for each edge $q\xrightarrow{e} q'$ of $\mathscr A$.\index{flow on an $L$-automaton}
\end{Def}

We need to consider complete set-partition flows for the case of $\sp$.

\begin{Def}[Complete flow on an automaton]
An $\mathsf {SP}(M,X)$-flow $F$ on a partial deterministic $X$-automaton $(Q,X)$ is
called a \emph{complete flow} if:
\begin{enumerate}
\item $F$ extends to an $\sp$-flow on $\mathscr A^{\square}$
via $\square F = B$;
\item $F$ is \emph{fully defined}, meaning, for each $r\in R$, there is a state $q\in Q$ such that $(r,r)\leq qF$.
\end{enumerate}\index{complete flow on an automaton}
\end{Def}

Conditions (1) and (2) are to guarantee that $F$ comes
from a relational morphism as we shall see in Section~\ref{s:present}.

For example, if $x\in X$, consider the partial automaton $\mathscr
A[x]$ given by \[q_0\xrightarrow x q_1.\]  Then
$F\colon \{q_0,q_1\}\to L$ is a flow if and only if $qF\xrightarrow {\ov
x}q'F$; that is, flows on $\mathscr A[x]$ correspond to
stable pairs for free flow along $x$.

Let $\mathscr A=(Q,\delta)$ be an $L$-automaton.  The set of flows on
$\mathscr A$ is denoted $FL(\mathscr A)$.  If $L=\sp$ and $\mathscr A$ is
a partial deterministic $X$-automaton, the set of  complete flows on $\mathscr
A$ is denoted $CFL(\mathscr A)$. We can view $FL(\mathscr A)$
and $CFL(\mathscr A)$ as subsets of the complete lattice $L^Q$ with
coordinate-wise operations.

\begin{Prop}\label{flowsareclosureops}
Let $\mathscr A=(Q,\delta)$ be an $L$-automaton. The set $FL(\mathscr
A)$ is a meet-closed subset of $L^Q$.
\end{Prop}
\begin{proof}
If $qF =T$ for all $q\in Q$, then $F$ is a flow since $T\xrightarrow{\ov
{e\delta}} T$ for all edges $e$.  This yields closure under empty meets. Let
$\{F_{\alpha}\colon Q\to L\mid \alpha\in A\}$ be a collection of flows. Then, for each edge
$q\xrightarrow {e} q'$, we have $qF_{\alpha}\xrightarrow{\ov {e\delta}}
q'F_{\alpha}$ for all $\alpha\in A$. Since $\ov{e\delta}$ is
meet-closed, we see that \[\bigwedge_{\alpha\in A} qF_{\alpha}\xrightarrow {\ov
{e\delta}} \bigwedge_{\alpha\in A} q'F_{\alpha}\] and so $\bigwedge
_{\alpha\in A}F_{\alpha}$ is a flow.
\end{proof}

It follows that $FL(\mathscr A)$ gives rise to a closure operator on
$L^Q$.  That is, if we label each vertex of $\mathscr A$ by an
element of $L$ via a function $f\colon Q\to L$, then there is a least flow
$\ov{\mathscr A}(f)\colon Q\to L$ on $\mathscr A$ such that $qf\leq
q\ov{\mathscr A}(f)$ for all $q\in Q$.  In general $CFL(\Ac)$ is
not meet-closed.  In fact, it almost never contains the empty meet.

\begin{Example}
Let us consider an example of a flow with $L$ the set
flow lattice. Suppose we have a pointed complete automaton $\Ac =
(Q,X)$ with a base point $q_0$ such that $q_0X^* = Q$.   Let $N$ be
the transition monoid of $\Ac$. Fix a base point $r_0\in R$.
Consider the smallest relational morphism~\cite{Eilenberg}
$\p\colon (R,M)\to (Q,N)$ such that $r_0$ relates $q_0$ and such that if
$m\in R$, $mx\in R$ and $m$ relates to $q$, then $mx$ relates to $qx$. That is, \[r\p = \{q_0\cdot w\mid w\in X^*\
\text{and}\ r=r_0\cdot w\}.\]  This is a fully-defined relation
since $\Ac$ is complete.  Define a function \mbox{$f\colon Q\to P(R)$}
that assigns $\{r_0\}$ to $q_0$ and $\emptyset$ to every other
vertex. Then $q\ov{\mathscr A}(f) = q\pinv$.
\end{Example}

In the same context, if we use the set-partition flow lattice,
then in addition to computing the relational morphism $\p$, we
will be computing the partitions giving rise to the minimal
injective automaton congruence on the derived transformation
semigroup~\cite{Eilenberg} (viewed as an automaton) of $\p$.   More details will
follow in the next section.

If $\Ac = (Q,\delta)$ is an $L$-automaton and $q\neq q'\in Q$, we can
obtain a new abstract flow on $L$ by \emph{sampling} at the states
$q,q'$.

\begin{Def}[Sampling at two states]\label{sample2}
Let $\Ac=(Q,\delta)$ be an $L$-au\-to\-ma\-ton and $q\neq q'\in Q$. The
abstract flow $\Ac(q,q')\in \mathscr C(L^2)$ is defined as follows.
Let $\ell,\ell'\in L$.  Define $f\colon Q\to L$ by $qf = \ell$, $q'f =
\ell'$ and by sending all other states to the bottom $B$ of $L$.
Then we define
\[\ell\xrightarrow {\Ac(q,q')} \ell' = (q \ov{\mathscr A}(f),q'\ov{\mathscr A}(f))\]
to be the result of sampling $\mathscr A$ at $q,q'$.\index{sampling at two states}
\end{Def}
In other words we consider all flows that are greater than or equal
to $\ell$ at $q$ and to $\ell'$ at $q'$, take their meet and then
sample the values at $q$ and $q'$.

\begin{Prop}\label{sampleisclosure}
$\Ac(q,q')$ is a closure operator.  Moreover, if $\ell,\ell'\in L$,
then $\ell\xrightarrow {\ov{\Ac(q,q')}}\ell'$ if and only if there is a
flow $F$ on $\Ac$ such that $qF = \ell$ and $q'F = \ell'$.
\end{Prop}
\begin{proof}
It is straightforward to verify that $\Ac(q,q')$ is order-preserving and increasing.  To see that it is idempotent, suppose that $\ell_1 \flow{\Ac(q,q')} \ell_2 = (\ell_1',\ell_2')$.  Define $f\colon Q\to L$ by $qf = \ell_1$, $q'f =
\ell_2$ and by sending all other states to the bottom $B$ of $L$; let $f'$ be defined analogously but with $qf'=\ell_1'$ and $q'f'=\ell_2'$.  Then $f\leq f'\leq \ov{\mathscr A}(f)$ and so $\ov{\mathscr A}(f)=\ov{\mathscr A}(f')$ and hence $\ell_1' \flow{\Ac(q,q')} \ell_2' = (\ell_1',\ell_2')$.  Thus $\Ac(q,q')$ is a closure operator.

We next prove the second statement, describing the image $\ov {\Ac(q,q')}$ of $\Ac(q,q')$. Let $\ell,\ell'\in L$ and
define the map \mbox{$f\colon Q\to L$} as in Definition~\ref{sample2}.   By
definition of $\Ac(q,q')$, if $\ell\xrightarrow {\ov{\Ac(q,q')}}\ell'$,
then
\begin{align*}
\ell &=q\ov{\mathscr A}(f)\\
\ell' &= q'\ov{\mathscr A}(f)
\end{align*}
Conversely, if $F$ is a flow with $qF=\ell$, $q'F=\ell'$, then $f\leq
F$ and therefore $\ov{\mathscr A}(f)\leq F$.  Thus
\begin{align*}
\ell &= qf \leq q\ov{\mathscr A}(f) \leq qF =\ell\\
\ell' &= q'f \leq q'\ov{\mathscr A}(f) \leq q'F =\ell'
\end{align*}
establishing the second statement.
%
% We now show that the set of stable pairs of $\Ac(q,q')$ is
% meet-closed. The case of the empty meet is trivial to deal with
% since the map sending each element to the top is a flow. Suppose
% $\ell_{\alpha}\xrightarrow {\ov {\Ac(q,q')}} \ell'_{\alpha}$ for $\alpha\in A$.  Then, for
% each $\alpha\in A$, there is a flow $F_{\alpha}$ such that
% $\ell_{\alpha}=qF_{\alpha}$ and $\ell'_{\alpha} = q'F_{\alpha}$.
% Hence
% \begin{align*}
% \bigwedge_{\alpha\in A} \ell_{\alpha} &= \bigwedge _{\alpha\in A} qF_{\alpha}
% =
% (q)\bigwedge_{\alpha\in A} F_{\alpha}\\
% \bigwedge_{\alpha\in A} \ell'_{\alpha} &= \bigwedge _{\alpha\in A}
% q'F_{\alpha} = (q')\bigwedge_{\alpha\in A} F_{\alpha}.
% \end{align*}
% The second statement of the proposition now implies that \[\bigwedge
% _{\alpha\in A}\ell_{\alpha}\xrightarrow {\ov {\Ac(q,q')}} \bigwedge
% _{\alpha\in A}\ell'_{\alpha},\] as required.
\end{proof}

As an example, let $x\in X$ and consider the automaton \[\Ac[x] =
q_0\xrightarrow x q_1.\]  It follows directly from the definition that the
closure operator $\Ac[x](q_0,q_1)$ is free flow along $x$ (that is
the operator $\xrightarrow x$).  Let us generalize this to free flow along
a word $w\in X^*$.  Suppose $w=x_1\cdots\ x_n$ and let \[\Ac[w] =
q_0\xrightarrow {x_1} q_1\xrightarrow {x_2}\cdots \xrightarrow {x_n} q_{n}.\]  Then
the closure operator $\Ac[w](q_0,q_n)$ is free flow along $w$.  This
follows from the following more general result.

\begin{Prop}\label{freeflowbysample}
Let $f_1,\cdots,f_n\in \mathscr C(L^2)$ and $f=f_1\cdots f_n$.
Consider the $L$-automaton
\begin{equation*}
\Ac [f_1,\ldots, f_n]=q_0\xrightarrow {f_1}q_1\xrightarrow {f_2}\cdots \xrightarrow
{f_n} q_{n}.
\end{equation*}
Then $\Ac [f_1,\ldots, f_n](q_0,q_n) = f$.
\end{Prop}
\begin{proof}
Observe that, by Proposition~\ref{sampleisclosure}, $\ell\xrightarrow
{\ov{\Ac[f_1,\ldots,f_n](q_0,q_n)}} \ell'$ if and only if there is a flow $F$ on
$\Ac[f_1,\ldots, f_n]$ with $q_0F = \ell$ and $q_nF = \ell'$.  This occurs if and
only if we can choose a function $F\colon Q\to L$ so that $q_0F = \ell$,
$q_nF =\ell'$ and \[q_{i-1}F\xrightarrow {\ov {f}_i} q_iF\] for all $i$. But
since $f$ is the product $f_1\cdots f_n$, it follows that such a
function $F$ exists if and only if $\ell\xrightarrow {\ov {f}} \ell'$.
\end{proof}

As another example, let $M$ and $N$ be $X$-generated monoids and put $L =P(M)$.  Define an
$X$-flow lattice structure on $L$ by $U\xrightarrow{\ov x} Y$ if and
only if $Ux\subseteq Y$.  Let $\Ac$ be the right Cayley graph of
$N$. Let $n\in N$. Then, for $1\neq n\in N$,
\[\{1\}\xrightarrow{\Ac(1,n)} \emptyset = (1\pinv,n\pinv)\] where
$\p\colon M\to N$ is the canonical relational morphism respecting the
generators $X$, that is, the relational morphism whose graph is
the submonoid generated by the image of the diagonal map $X\to
M\times N$.

We can also get a one-variable closure operator by sampling at a
state.

\begin{Def}[Sampling at a state]
Let $\Ac=(Q,\delta)$ be an $L$-automaton and $q\in Q$. Then an element
$\Ac(q)\in \mathscr C(L)$  can be defined as follows. Let $\ell\in
L$. Define $f\colon Q\to L$ by $qf = \ell$ and by sending all other states
to $B$.  Then we define $\ell\Ac(q) = q\ov{\mathscr A}(f)$.\index{sampling at one state}
\end{Def}

One can verify that $\Ac(q)$ is a closure operator in a similar
fashion to Proposition~\ref{sampleisclosure}.  Many of our one-variable closure operators can be interpreted via sampling, as the following proposition, whose proof is merely unwinding the definitions, shows.

\begin{Prop}
Let $f\in \mathscr C(L^2)$.  Then:
\begin{enumerate}
\item $\overleftarrow f = \Ac[f](q_0)$ where $\Ac[f] = q_0\xrightarrow{f}q_1$;
\item $f^* =\mathscr A[f^*](q)$ where $\mathscr A[f^*] = q\looop f$.
\end{enumerate}
\end{Prop}

This proposition should explain the intuition behind the names back-flow and the Kleene star.  Let's give further motivation for the star notation via an example. Let $w\in X^*$.  Let $M$ be a finite $X$-generated group
mapping monoid with distinguished $\R$-class $R$ and consider the
set flow lattice on $R$.  Let $U\in P(R)$.  We claim that, for $w\in X^*$, \[U\looop
{w} = Uw^*\] where $w^*$ is the submonoid generated by the image of
$w$ in $M$. Indeed, \[U\looop w\] is, by definition, the least subset
$Y$ containing $U$ such that $Yw\subseteq Y$.  But this is exactly
$Uw^*$. Intuitively, the one-variable operator \[\looop w\] is obtained by taking the automaton
$\Ac[w]$, identifying $q_0$ with $q_n$ and then sampling at $q_0$.
Notice that the language of the resulting automaton is $w^*$.

If $f,g\in \mathscr C(L^2)$, we can define an $L$-automaton
\[\xymatrix{q_0\ar[r]^f &q_1\ar@(r,u)_g}.\]  The two variable closure operator
obtained by sampling with respect to $(q_0,q_1)$ is none other than $fg^*$.  Unwinding
the definition we see that, for $\ell,\ell'\in L$, we have $\ell\flow{\ov{fg^*}} \ell'$, drawn \[\ell\plstar {\ov f}{\ov g}\ell',\] if and only if $\ell\xrightarrow {\ov f}\ell'$ and
$\ell'\xrightarrow {\ov g}\ell'$. The picture indicates that you flow from
$\ell$ to $\ell'$ via $f$ and then flow in a loop from $\ell'$ to
$\ell'$ along $g$.

For instance, consider set flows on a finite $X$-generated group
mapping monoid $M$ with distinguished $\R$-class $R$.  Let
$U,Y\subseteq R$ and $x,y\in X$.  Then $U\xrightarrow{xy^*} Y = (U,(Ux\cup
Y)y^*)$.

\begin{Prop}\label{cheapbound}
Let $f,g\in \mathscr C(L^2)$.  Then:
\begin{enumerate}
\item $f\leq f\overleftarrow g\leq fg^*$;
\item For all $k\geq 0$, $g^k\leq \bigvee_{m\geq 0} g^m\leq g^*$;
\item $f\leq \bigvee_{k\geq 0}(fg^k)\leq f\left(\bigvee_{k\geq 0} g^k\right)\leq fg^*$;
\item The map $g\mapsto g^*$ is a closure operator on $\mathscr C(L^2)$ and in particular $(g^*)^*=g^*$;
\item $(f\vee g)^*=f^*g^*=f^*\vee g^*=g^*f^*$;
\item $(fg^*)^*=f^*\vee g^*=(f^*g)^*$.
\end{enumerate}
\end{Prop}
\begin{proof}
Proposition~\ref{actionfacts} establishes the first inequality of (1).  The second follows since $\overleftarrow{g}\leq g^*$. The first inequality of (2) is clear.  For the second, a stable pair for $g^*$ looks like $(\ell,\ell)$ where $\ell\xrightarrow{\ov g}\ell$.  But then $\ell\xrightarrow{\ov {g^m}}\ell$ for all $m\geq 0$ and hence $(\ell,\ell)$ is a stable pair for $\bigvee_{m\geq 0}g^m$.
Item (3) is an immediate consequence of (2) and the fact that $\mathscr C(L^2)$ is an ordered monoid. The fourth item is trivial.

For (5), note that $1_{\fix {\ov f}}1_{\fix {\ov g}}=1_{\fix{\ov f}\cap \fix{\ov{g}}} = 1_{\fix{\ov{f\vee g}}}$ and so $f^*g^*=(f\vee g)^*$.  The remaining equalities follow from Proposition~\ref{twoasone}.  Finally, for (6) we have by the previous parts that \[(fg^*)^*\leq (f^*g^*)^* = ((f\vee g)^*)^*=(f\vee g)^*=f^*\vee g^*.\]  Conversely, suppose that $\ell\in \fix {\ov{fg^*}}$.  Then \[\ell \plstar{\ov f}{\ov g}\ell\] and hence $\ell\in \fix {\ov f}\cap \fix {\ov g}$.  Thus $f^*\vee g^*\leq (fg^*)^*$ completing the proof of the first equality.  The second is dual.
\end{proof}

We remark that the inequalities of the proposition are in general
strict.

\subsubsection{Finite lattices}
Assume now that $L$ is a finite lattice.  Then $\mathscr C(L^2)$ is a finite monoid.
 If $s$ is an element of a profinite semigroup, then $s^{\omega}$ denotes the unique idempotent in $\ov{\langle s\rangle}$. For $k\geq 0$, $s^{\omega+k} = s^{\omega}s^k$.
For $k<0$, we denote by $s^{\omega-k}$ the inverse of $s^{\omega+k}$
in the procyclic group $\ov{\langle s^{\omega+1}\rangle}$.

\begin{Def}[$()^{\omega+\ast}$]
If $L$ is a finite lattice, and $f\in \mathscr C(L^2)$, then we set
\begin{equation}\label{defomega+star}
f^{\omega+\ast}=f^{\omega}f^*.
\end{equation}
In other words, $(\ell,\ell')$ is stable for $f^{\omega+\ast}$ if and only if
\[\ell \plstar {\ov {f^{\omega}}}{\ov f}\ell',\] if and only if $\ell\xrightarrow{\ov {f^{\omega}}}\ell'\xrightarrow{\ov f}\ell'$.\index{$()^{\omega+\ast}$}
\end{Def}

\begin{Rmk}
Notice that $f^{\omega}\leq f^{\omega+\ast}$ by Proposition~\ref{actionfacts}.
\end{Rmk}

We establish a few basic properties of $f^{\omst}$.

\begin{Prop}\label{gomega+*isidempotent}
Let $g\in \mathscr C(L^2)$.  Then $g^{\omega+\ast}$ and $g^{\omst}
g^{\omega}$ are $\R$-equivalent idempotents.  Moreover, $g^{\omst}
g^{\omega}\leq g^{\omst}$.
\end{Prop}
\begin{proof}
Since $g^{\omega}\leq g^*$ by Proposition~\ref{cheapbound}, evidently $g^{\omst}g^{\omega}\leq  g^{\omega}g^*g^* = g^{\omega}g^*=g^{\omst}$, establishing the second statement.

To prove the first statement we compute \[g^{\omst} = g^{\omega}g^{\omst}\leq g^{\omst}g^{\omst}=g^{\omst}g^{\omega}g^*\leq g^{\omst}g^*=g^{\omst}.\]  Therefore, $g^{\omst}g^{\omega}g^{\omst} = g^{\omst}g^{\omst}=g^{\omst}$ and $g^{\omst}g^{\omst}g^{\omega} = g^{\omst}g^{\omega}$, whence $g^{\omst}\R g^{\omst}g^{\omega}$.  Also, $g^{\omst}g^{\omega}g^{\omst}g^{\omega} = g^{\omst}g^{\omega}$.
\end{proof}

The following lemma shows that $g^{\omega+\ast}$ and its dual absorb $g$.

\begin{Lemma}\label{rhodestrickiness}
Let $L$ be a finite lattice and $g\in \mathscr C(L^2)$.  Then
$gg^{\omega+\ast} = g^{\omst}$ and dually $g^*g^{\omega}g=g^*g^{\omega}$. Hence
\begin{equation}\label{absorb}
g^{\omega+\ast}g^{\omega}g =
g^{\omega+\ast}g^{\omega}=gg^{\omega+\ast}g^{\omega}.
\end{equation}
\end{Lemma}
\begin{proof}
Equation \eqref{absorb} is an immediate consequence of the first part of the lemma.  We just prove $gg^{\omst} = g^{\omst}$ as the other equality is dual.  Indeed, Proposition~\ref{cheapbound} yields
\begin{align*}
gg^{\omst}&=gg^{\omega}g^* = g^{\omega}gg^* \\
& \leq g^{\omega}g^*g^*=g^{\omst} = g^{\omega+1}g^{\omega-1}g^*\\
&\leq g^{\omega+1}g^*g^* = gg^{\omega}g^* = gg^{\omst}
\end{align*}
as required.
\end{proof}

When trying to establish the companion upper bound to the lower bound introduced in this paper, it will often be necessary to work with a ``conjugated'' version of $f^{\omst}$.

\begin{Thm}[Conjugated star]
Let $f,g\in \mathscr C(L^2)$ and consider the $L$-automaton $\mathscr A$:
\[\xymatrix{q_0\ar[r]^{(fg)^{\omega}}& q_1\ar@/_/[r]_f &\ar@/_/[l]_gq_2}.\]
Then $\mathscr A(q_0,q_2) = f(gf)^{\omst}$.
\end{Thm}
\begin{proof}
Suppose first that $(\ell,\ell')$ is stable for $\mathscr A(q_0,q_2)$.  Then by Proposition~\ref{sampleisclosure} we can find $\ell''\in L$ so that \[\xymatrix{\ell\ar[r]^{\ov{(fg)^{\omega}}}& \ell''\ar@/_/[r]_{\ov f} &\ar@/_/[l]_{\ov g}\ell'}\]
and hence we can find $\ell_0,\ell_1\in L$ such that \[\ell\xrightarrow{\ov f}\ell_0\xrightarrow{\ov g}\ell_1\xrightarrow{\ov{(fg)^{\omega-1}}}\ell''\xrightarrow{\ov f}\ell'\xrightarrow{\ov {gf}}\ell'.\]
Composing, we obtain \[\xymatrix{\ell\ar[r]^{\ov f}&\ell_0\ar[r]^{\ov{(gf)^{\omega}}}&\ell'\ar@(r,u)_{\ov{gf}}}\]
and so $\ell\flow{\ov{f(gf)^{\omst}}}\ell'$, as required.

Conversely, assume $(\ell,\ell')$ is stable for $f(gf)^{\omst}$.  Lemma~\ref{rhodestrickiness} yields $f(gf)^{\omst}=f(gf)^{\omega-1}(gf)^{\omst}$ and so we can find $\ell_1,\ell_2\in L$ such that
\[\ell\xrightarrow{\ov {f(gf)^{\omega-1}}}\ell_1\xrightarrow{\ov {(gf)^{\omega}}}\ell'\xrightarrow{\ov g}\ell_2\xrightarrow{\ov f}\ell'.\] Because $f(gf)^{\omega-1}(gf)^{\omega}g=(fg)^{\omega}$,  we conclude
\[\xymatrix{\ell\ar[r]^{\ov{(fg)^{\omega}}}& \ell_2\ar@/_/[r]_{\ov f} &\ar@/_/[l]_{\ov g}\ell'}\]
and so $\ell\xrightarrow{\ov{\mathscr A(q_0,q_2)}}\ell'$, again by Proposition~\ref{sampleisclosure}, completing the proof.
\end{proof}

\section{The Presentation Lemma: Flow Form}\label{s:present}
The main tool for dealing with complexity is the Presentation
Lemma.  We shall use the version of~\cite[Section 4.14]{qtheor} (see also~\cite{aperiodic}), rather than that of~\cite{pl}.  The key difference is that~\cite{pl} views $R$ as $G\times B$ where $G$ is the maximal subgroup of $R$ and $B$ is the set of $\H$-classes of $R$ and uses the Dowling lattice (which was invented by the second author in 1968 before Dowling, but only published much later in~\cite{pl}) instead of the set-partition lattice.   The goal of this section is to prove the following result, where $\malce$ denotes the Mal'cev product and $\mathsf{RLM}(M)$ is the right letter mapping image of $M$ (see below or~\cite{qtheor}).

\begin{Thm*}[Presentation Lemma: Flow form]
Let $M$ be a finite $X$-generated group mapping monoid with
distinguished $\R$-class $R$.  Let \pv V be a pseudovariety.  Then
$M\in \pv A\malce (\pv G\ast \pv V)$ if and only if $\mathsf{RLM}(M)\in \pv
A\malce (\pv G\ast \pv V)$ and there exists a complete $\mathsf
{SP}(M,X)$-flow $F$ on a partial automaton $\Ac$
over $X$ with transition monoid in $\pv V$ such that if $r,s$ are in the same
block of $qF$ for some state $q$ and $r\H s$, then $r=s$.
\end{Thm*}

The aim of this section is to show that the statement of the above theorem is equivalent to the Presentation Lemma as stated in~\cite[Theorem 4.14.19]{qtheor}.  The reader who is willing to accept
this as a fact may skip ahead to Theorem~\ref{flowform}.

In this paper, we mean by a transformation semigroup a faithful
partial transformation semigroup as per~\cite{Eilenberg}.   A
\emph{relational morphism} of partial transformation semigroups \mbox{$\p\colon (R,M)\to (Q,N)$}  is a fully defined
relation $\p\colon R\to Q$ such that, for all $m\in M$, there exists
$\til m\in N$ so that
\begin{equation}\label{cover2}
q\pinv m\subseteq (q\til m)\pinv
\end{equation}
for all $q\in Q$.  One says in this case that $\til m$
\emph{covers} $m$. There is a \emph{companion relational morphism}
$\p^c\colon M\to N$ defined by
\[m\p^c = \{n\in N\mid n\ \text{covers}\ m\}.\]

A \emph{parameterized relational morphism}\index{parameterized relational morphism} of partial transformation semigroups $\Phi\colon (R,M)\to (Q,N)$
is a pair $(\p_0,\p_1)$ where $\p_0\colon (R,M)\to (Q,N)$ and $\p_1\colon M\to
N$ are relational morphisms such that $\p_1\subseteq \p_0^c$, that
is, each $n\in m\p_1$ covers $m$.  Suppose that $M$ and $N$ are
both $X$-generated.  Then the parameterized relational morphism is
termed \emph{canonical} if $\p_1$ is the relational morphism whose
graph is generated by all pairs of the form $([x]_M,[x]_N)$ with
$x\in X$.  Here we use the convention that if $M$ is an
$X$-generated monoid and $w\in X^*$, then $[w]_M$ is the image of
$w\in M$.  Sometimes, we just write $w$ if $M$ is understood.

Let $\Phi\colon (R,M)\to (Q,N)$ be a parameterized relational morphism.
We shall need the following partial automaton, denoted
$\D_{\Phi}$, which is in fact the \emph{derived transformation
semigroup} of $\Phi$~\cite{Eilenberg} (without the empty function)
viewed as an automaton. The state set of $\D_{\Phi}$ is
\[\# \p_0 = \{(r,q)\mid q\in r\p_0\}.\]  The transitions
are of the form \[(r,q)\xrightarrow {(q,(m,n))} (rm,qn)\] where $n\in
m\p_1$ and $rm\in R$.

By an \emph{automaton congruence}\index{automaton congruence} on a partial automaton $\mathscr
A = (Q,X)$ we mean an equivalence relation $\equiv$ on $Q$ such
that $q\equiv q'$ and $qx,q'x\in Q$ implies $qx\equiv q'x$, for
$q,q'\in Q$, $x\in X$. The quotient automaton $\mathscr
A/{\equiv}$ has state set $Q/{\equiv}$ and input alphabet $X$.
There is a transition $[q]\xrightarrow x [q']$ if and only if there are
$q_0,q_0'\in Q$ with $q_0\in [q]$, $q_0'\in [q']$ and $q_0\xrightarrow
x q_0'$.  The automaton congruence is called \emph{injective} if
$qx,q'x\in Q$ and $qx\equiv q'x$ implies $q\equiv q'$, that is,
the transitions of the quotient automaton $\mathscr A/{\equiv}$
are partial one-to-one.

An automaton congruence, and its associated partition, on $\D_{\Phi}$ is called \emph{admissible}
if it is injective and \[(r,q)\equiv (r',q')\implies q=q'\]
for $r,r'\in R$ and $q,q'\in Q$.

Suppose now that $M$ is a finite group mapping monoid with
distinguished $\R$-class $R$.  Let $B$ be the set of $\mathscr L$-classes
of $M$.  Then $M$ acts by partial transformations on $B$ via right
multiplication, resulting in a transformation semigroup
$(B,\mathsf{RLM}(M))$.  Following~\cite{Arbib,qtheor}, $\mathsf{RLM}(M)$ is called the
\emph{right letter mapping semigroup of $M$}\index{right letter mapping}.  Since $R$ contains a
non-trivial group, $\mathsf{RLM}(M)$ is always a proper image of $M$~\cite{Arbib,qtheor,pl}.

\begin{Def}[Presentation] Let \pv V be a pseudovariety of monoids. Then a
\emph{presentation} for $(R,M)$ over $\pv V$ is a pair
$(\Phi,\mathscr P)$ where $\Phi\colon (R,M)\to (Q,N)$ is a parameterized
relational morphism, $N\in \pv V$ and $\mathscr P$ is an admissible
partition on $\D_{\Phi}$ such that
\[(r,q)\mathrel{\mathscr P} (s,q)\ \text{and}\ r\H s\implies r=s\] for $r,s\in R$ and $q\in Q$.\index{presentation}
\end{Def}

The following result is the Presentation Lemma~\cite[Theorem 4.14.19]{qtheor}, originally due to the second author~\cite{pl}.
Recall that \pv A denotes the pseudovariety of aperiodic monoids
and \pv G denotes the pseudovariety of finite groups.  If \pv V
and \pv W are pseudovarieties, $\pv V\ast \pv W$ denotes their
semidirect product~\cite{Eilenberg,qtheor} and $\pv V\malce \pv
W$ their Malcev product~\cite{HMPR,qtheor}.

\begin{Thm}[Presentation Lemma]\label{preslem}
Let $M$ be a finite group mapping monoid and \pv V be a
pseudovariety. Then $M\in \pv A\malce (\pv G\ast \pv V)$ if and
only if $\mathsf{RLM}(M)\in \pv A\malce (\pv G\ast \pv V)$ and $(R,M)$ has
a presentation over $\pv V$ where $R$ is the distinguished
$\R$-class of $M$.
\end{Thm}

Let $\pv C_n$ denote the pseudovariety of monoids of complexity at
most $n$~\cite{Eilenberg,qtheor}.  The Fundamental Lemma of Complexity~\cite{flc,TilsonXII} shows that \[\pv C_n = \pv A\malce (\pv G\ast
\pv C_{n-1}).\]   It is also a well-known consequence of the
Fundamental Lemma of Complexity that the decidability of
complexity reduces to the case of group mapping monoids; see the
discussion in~\cite{pl}. The above theorem, with $\pv V = \pv
C_{n-1}$, shows that decidability of complexity $n$ reduces to the
decidability of the existence of presentations over $\pv C_{n-1}$
for group mapping monoids.  Here we are using the fact that we can
assume by induction on order that membership of $\mathsf{RLM}(M)$ in $\pv C_n$ can
already be determined.

There is also a stronger version of Theorem~\ref{preslem}.  Recall
that if $M$ is a monoid and $\pv V$ is a pseudovariety, then a
subset $A\subseteq M$ is called \emph{$\pv V$-pointlike} if, for all
relational morphisms $\p\colon M\to N$ with $N\in \pv V$, there exists
$n\in N$ with $A\subseteq n\p\inv$~\cite{qtheor}.  The Presentation Lemma for pointlikes is due to the third author~\cite{aperiodic} and is~\cite[Theorem 4.14.20]{qtheor}.

\begin{Thm}[$\pv G\ast \pv V$-pointlikes]\label{pointlikes}
Let $M$ be a finite group mapping monoid with distinguished
$\R$-class $R$.  Let $\pv V$ be a pseudovariety.  Then $A\subseteq
R$ is $\pv G\ast \pv V$-pointlike if and only if, for every
parameterized relational morphism $\Phi\colon (R,M)\to (Q,N)$ with $N\in
\pv V$ and every admissible partition $\mathscr P$ on
$\D_{\Phi}$, there exists $q\in Q$ such that $A\subseteq
q\p_0\inv$ and $A\times \{q\}$ is contained in a single block of
$\mathscr P$.
\end{Thm}

We now wish to show how to go between set-partition flows and
parameterized relational morphisms with admissible
partitions on their derived automata.

\begin{Prop}\label{correspondence}
Let $M$ be an $X$-generated group mapping monoid with distinguished $\R$-class $R$.
\begin{enumerate}
\item Suppose $\mathscr A=(Q,X)$ is a partial automaton with transition monoid $N$ and let $F\colon Q\to \sp$ be a complete flow on $\mathscr A$.  Then there exist a parameterized relational morphism $\Phi\colon (R,M)\to (Q,N)$ and an admissible partition $\Pc$ on $\D_{\Phi}$ so that if $qF=(A,P)$, then $q\p_0\inv = A$ and $(r,q)\Pc (r',q)$, for $r,r'\in A$, if and only if $r\flow{P} r'$.
\item If $\Phi\colon (R,M)\to (Q,N)$ is a parameterized relational morphism and $\Pc$ is an admissible partition on $\D_{\Phi}$, then we can find a partial automaton $\mathscr A=(Q,X)$ with transition monoid a submonoid of $N$ and a complete flow $F$ on $\mathscr A$ so that, for all $q\in Q$, $qF=(q\p_0\inv, P_q)$ where $r\flow{P_q} r'$ if and only if $(r,q)\Pc (r',q)$.
\end{enumerate}
\end{Prop}
\begin{proof}
We begin with (1).  Define a canonical parameterized relational morphism $\Phi\colon (R,M)\to (Q,N)$ by putting $q\p_0\inv=A$ where $qF=(A,P)$.
To see that $\p_0$ is fully defined, let $r\in R$.  Since $F$ is a complete flow, there exists $q$ with $(r,r)\leq qF$.  Then $r\in q\p_0\inv$.  Next, we show that $[x]_N$ covers $[x]_M$.  Indeed, suppose $q\in Q$ and $x\in X$.  Let $qF=(A,P)$.    Assume first that $qx=\emptyset$.  Then since $F$ is a complete flow, we must have $Ax=\emptyset$.  Thus $q\p_0\inv [x]_M = Ax=\emptyset\subseteq q[x]_N\p_0\inv$.  Next suppose $qx\neq\emptyset$.  Let $(qx)F = (A',P')$.  Then, since $F$ is a flow, $Ax\subseteq A'$.  Hence $q\p_0\inv [x]_M = Ax\subseteq A'= (q[x]_N)\p_0\inv$.  Thus $[x]_N$ covers $[x]_M$, from which it follows by an easy induction that $[w]_N$ covers $[w]_M$ all $w\in X^*$.  So $\Phi$ is indeed a canonical parameterized relational morphism.

Next define a partition $\Pc$ on $\D_{\Phi}$ by setting $(r,q)\Pc (r',q)$ if $qF=(A,P)$ with $r,r'\in A$ and $r\flow {P} r'$.
Our goal is to verify that $\Pc$  is an admissible partition.  It is immediate that $\Pc$ is a partition. To see  it is an automaton congruence, we prove by induction on length that if $w\in X^*$ and \mbox{$(r,q)\Pc (r',q)$} are such that $(r,q)(q,([w]_M,[w]_N))$ and $(r',q)(q,([w]_M,[w]_N))$ are defined, then $(r[w]_M,q[w]_N)\Pc (r'[w]_M,q[w]_N)$.  This is trivial if $|w|=0$.  Suppose it is true for all words of length at most $n$ and suppose $w=ux$ with $|u|=n$ and $x\in X$.   By induction, $(r[u]_M,q[u]_N)\Pc (r'[u]_M,q[u]_N)$.     Suppose $quF =(A,P)$ and $qwF=quxF=(A',P')$.  Then $r[u]_M\flow{P} r'[u]_M$.   Since $F$ is a flow, $Ax\subseteq A'$, whence $r[u]_Mx,r'[u]_Mx\in A'$ and $r[u]_Mx\flow{P'}r'[u]_Mx$.  Thus $(r[w]_M,q[w]_N)\Pc (r'[w]_M,q[w]_N)$ and so $\Pc$ is an automaton congruence.

To see that $\Pc$ is injective, we establish by induction on length that if
$w\in X^*$, $(q_0,([w]_M,[w]_N))$ is defined on $(r,q), (r',q')$ and
\[(r,q)(q_0,([w]_M,[w]_N))\Pc (r',q')(q_0,([w]_M,[w]_N)),\] then $(r,q)\Pc (r',q')$. First note that we must have
$q=q_0=q'$.  So if $|w|=0$, there is nothing to prove.  Suppose the claim is true for all words of length
at most $n$ and consider $w=ux$ with $x\in X$ and $|u|=n$.  Set $p=qu$.  Then we have
\begin{align*}
(r[u]_M,p)(p,([x]_M,[x]_N)) &= (r,q)(q_0,([w]_M,[w]_N))\\
&\Pc (r',q')(q_0,([w]_M,[w]_N))\\ &=(r'[u]_M,p)(p,([x]_M,[x]_N)).
\end{align*}
Let $(A,P) = pF$ and $(A',P')=pxF$.  As $F$ is a complete flow, $Ax\subseteq A'$ and
\begin{equation}\label{ispartinj}
A/P\xrightarrow{\,\cdot x\,}A'/P'\ \text{is partial injective}.
\end{equation}
Thus $r[ux]_M,r'[ux]_M\in A'$.  Because \mbox{$(r[w]_M,q[w]_N)\Pc (r'[w]_M,q'[w]_N)$}, it follows from the definition that $r[ux]_M=r[w]_M\flow{P'} r'[w]_M=r'[ux]_M$.  Therefore, $r[u]_M\flow{P} r'[u]_M$ by \eqref{ispartinj} and hence $(r[u]_M,p)\Pc (r'[u]_M,p)$.  Induction now yields $(r,q)\Pc (r',q')$, as required.  So $\Pc$ is an injective automaton congruence.  It is admissible directly from the definition, establishing (1).

For (2), suppose that $\Phi=(\p_0,\p_1)$. Fix, for each $x\in X$ an element $n_x\in [x]_M\p_1$.  Define a partial deterministic automaton $\mathscr A$ with state set $Q$ and transitions given by
$q\xrightarrow{x} qn_x$; the transition monoid of $\mathscr A$ is \[\langle n_x\mid x\in X\rangle\leq N.\]
Define a flow $F\colon Q\to \sp$ by setting $qF=(q\p_0\inv, P_q)$ where $r\flow{P_q} r'$ if and only if $(r,q)\Pc (r',q)$. To see that $F$ is fully-defined, observe that if $r\in R$, then there exists $q\in Q$ so that $r\in q\p_0\inv$.  Then $(r,r)\leq qF$.  Suppose that $qx=\emptyset$ and let $(A,P)=qF$.  Then \[Ax=q\p_0\inv x\subseteq qn_x\p_0\inv =qx\p_0\inv =\emptyset\] and so $Ax=\emptyset$, whence $qF\xrightarrow{\ov x}B$.  Thus $F$ is complete.  It remains to verify that $F$ is a flow.  Suppose that $q\xrightarrow{x}qx$ is a transition.  Then we have
\begin{equation}\label{randomeq}
q\p_0\inv x\subseteq qn_x\p_0\inv = qx\p_0\inv
\end{equation}
since $n_x\in [x]_M\p_1\inv$ and $\Phi$ is a parameterized relational morphism.

To ease notation set $A=q\p_0\inv$ and $A'=qx\p_0\inv$.  Then \eqref{randomeq} implies that $Ax\subseteq A'$.
Let $\pi\colon \# \p_0\to R$ and $\rho\colon \D_{\Phi}\to
\D_{\Phi}/{\Pc}$ be the projections.  There then results the commutative diagram in Figure~\ref{acomdiagram}
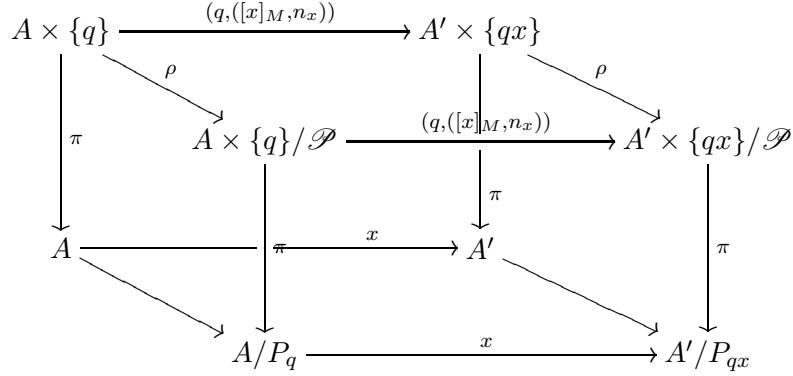
\begin{figure}[h!b!p!t!]
\SelectTips{eu}{12}
\[ \xymatrix{ A\times\{q\} \ar[dd]^{\pi} \ar[rd]^{\rho} \ar[rr]^{(q,([x]_M,n_x))} && A'\times \{qx\} \ar'[d][dd]^{\pi} \ar[rd]^{\rho} \\
& A\times\{q\}/{\Pc} \ar[dd]^{\pi} \ar[rr]^{(q,([x]_M,n_x))} && A'\times \{qx\}/{\Pc} \ar[dd]^{\pi} \\
A \ar'[r][rr]^{x} \ar[rd] && A' \ar[rd] \\
& A/P_q \ar[rr]^{x} && A'/{P_{qx}} }
\]
\caption{A commutative diagram showing $F$ is a flow\label{acomdiagram}}
\end{figure}
where the vertical lines are bijections and the unlabelled arrows are the projections.  Since $\Pc$ is an injective automaton congruence, it now follows that $A/P_q\xrightarrow{\cdot x}A'/P_{qx}$ is a partial injective map.  This completes the proof that $F$ is a flow thereby establishing (2).
\end{proof}

The above proposition easily leads to the following two theorems, which are restatements of
Theorems~\ref{preslem} and~\ref{pointlikes} in the language of flows.

\begin{Thm}[Presentation Lemma: Flow form]\label{flowform}
Let $M$ be a finite $X$-generated group mapping monoid with
distinguished $\R$-class $R$.  Let \pv V be a pseudovariety.  Then
$M\in \pv A\malce (\pv G\ast \pv V)$ if and only if \[\mathsf{RLM}(M)\in \pv
A\malce (\pv G\ast \pv V)\] and there exists a complete $\mathsf
{SP}(M,X)$-flow $F$ on a partial automaton $\Ac$
over $X$ with transition monoid in $\pv V$ such that if $(\{r,s\},\{\{r,s\}\})\leq qF$ for some state $q$ (i.e., $r,s$ belong to the same block of $qf$) and $r\H s$, then $r=s$.
\end{Thm}
\begin{proof}
Suppose first that such a flow $F$ exists.  Then $(\Phi,\Pc)$ constructed in Proposition~\ref{correspondence}(1) is a presentation over $\pv V$.  Conversely, if $(\Phi,\Pc)$ is a presentation over $\pv V$, then Proposition~\ref{correspondence}(2) defines the desired flow.  The result now follows from Theorem~\ref{preslem}.
\end{proof}

The next theorem can be deduced from Theorem~\ref{pointlikes} in an analogous fashion; we omit the proof.

\begin{Thm}[$\pv G\ast \pv V$-pointlikes: Flow form]\label{flowpoint}
Let $M$ be a finite $X$-generated group mapping monoid with
distinguished $\R$-class $R$.  Let \pv V be a pseudovariety.  Then
$A\subseteq R$ is $\pv G\ast \pv V$-pointlike if and only if, for
every complete $\mathsf {SP}(M,X)$-flow on a partial automaton $\Ac$
over $X$ with transition monoid in $\pv V$, there is a state $q$
such that $(A,\{A\})\leq qf$, i.e., $A$ is contained in a single block of $qf$.
\end{Thm}

So computing membership in $\pv C_n$ amounts to studying
set-partition flows on automata with transition monoid in $\pv
C_{n-1}$.

\subsection{Inevitable set-partitions}
In this section, we define the notion of an inevitable set-partition with respect to a pseudovariety $\pv V$. Throughout this section, we put $L=\sp$.

\begin{Def}[\pv V-inevitability]
Let \pv V be a pseudovariety. Then an element $\ell\in L$ is said to be \emph{\pv
V-inevitable} if, for all complete flows $F$ on a partial
automaton $\Ac=(Q,X)$ with transition monoid in $\pv V$, there is a
state $q\in Q$ such that $\ell\leq qF$.\index{$\pv V$-inevitability}
\end{Def}

  Notice the set of $\pv V$-inevitable elements of $L$
is an order ideal.  That is, if $\ell'\leq \ell$ and $\ell$ is $\pv
V$-inevitable, then so is $\ell'$.  Also notice that the points of $L$
are $\pv V$-inevitable by definition of a complete flow.  The importance of this notion comes from the following reformulation of Theorem~\ref{flowpoint}.

\begin{Cor}\label{pointlike-flowform}
Let $M$ be a finite $X$-generated group mapping monoid with
distinguished $\R$-class $R$.  Let \pv V be a pseudovariety.  Then
$A\subseteq R$ is \mbox{$\pv G\ast \pv V$}-pointlike if and only if there
is a $\pv V$-inevitable element \mbox{$(Y,P)\in \mathsf {SP}(L)$} such
that $A\subseteq Y$ and $A$ is contained in a block of $P$, i.e., $(A,\{A\})\leq (Y,P)$.
\end{Cor}
\begin{proof}
Theorem~\ref{flowpoint} says exactly that $A$ is $\pv G\ast \pv V$-pointlike if and only if $(A,\{A\})$ is $\pv V$-inevitable.  Since the set of $\pv V$-inevitable set-partitions is an order ideal, this completes the proof.
\end{proof}

 The following standard compactness result
is called the ``Little Boxes Theorem'' by the second author (the boxes
refer to the blocks of the partition in a set-partition
flow).

\begin{Thm}[Little Boxes Theorem]\label{littleboxes}
Let \pv V be a pseudovariety. Then there is a partial
automaton $\Ac=(Q,X)$ with transition monoid in $\pv V$ and a
complete flow $F$ on $\Ac$ such that $\ell\in L$ is $\pv
V$-inevitable if and only if $\ell\leq qF$ for some $q\in Q$.
\end{Thm}
\begin{proof}
Since $L$ is finite, the set $W$ of elements of $L$ that are not
$\pv V$-inevitable is finite.  For each $\ell\in W$, we can find a
partial automaton $\Ac_{\ell}=(Q_\ell,X)$ with transition monoid in
$\pv V$ and a complete flow $F_\ell$ on $\Ac_{\ell}$ such that
$\ell\nleq qF_{\ell}$ for all $q\in Q$. Let $\Ac=\prod_{\ell\in W}
\Ac_{\ell}=(Q,X)$. Then $\Ac$ has transition monoid in $\pv V$.
Define $F\colon Q\to L$ by
\begin{equation}\label{definelittlebox}
(q_\ell)F = \bigwedge _{\ell\in W} q_{\ell}F_{\ell}.\end{equation}
Suppose that $(q_\ell)\xrightarrow x (q'_\ell)$ is an edge of $\Ac$. Then,
for each $\ell\in L$, $q_\ell\xrightarrow x {q'_\ell}$.  So
$q_{\ell}F_\ell\xrightarrow {\ov x} q_\ell'F_\ell$ for all $\ell\in W$.  Since
$\xrightarrow {\ov x}$ is meet-closed, it follows $(q_\ell)F\xrightarrow {\ov x}
(q'_\ell)F$ and hence $F$ is a flow.

We must now check that $F$ is a complete flow.  To see that $F$ is
fully defined, let $r\in R$.  Then, for each $\ell\in W$,
there is a state $q_\ell$ with $(r,r)\leq q_{\ell}F_\ell$ since the $F_\ell$
are complete flows. Hence
\[(r,r)\leq \bigwedge _{\ell\in W} q_{\ell}F_{\ell} = (q_\ell)F,\]
establishing that $F$ is fully defined.

 To see that $F$ extends to $\Ac^{\square}$, we must show that if $(q_\ell)x$ is not defined, then
$(q_\ell)F\xrightarrow {\ov x} B$. But if $(q_\ell)x$ is not defined, then
$q_{\ell'}x$ is undefined for some $\ell'\in L$.  Hence
\begin{equation}\label{fl2}
q_{\ell'} F_{\ell'}\xrightarrow {\ov x} B
\end{equation}
since $F_{\ell'}$ is a complete flow.   Suppose $q_{\ell'}F_{\ell'}=(Y',P')$ and $(q_{\ell})F = (Y,P)$.  From the definitions, we have $(Y,P)\leq (Y',P')$.  But $Y'x=\emptyset$ by \eqref{fl2} and so $Yx=\emptyset$.  Therefore, $(q_{\ell})F\xrightarrow{\ov x} B$, establishing that $F$ is a complete flow.

To see that $F$ has the desired property, we must show that if
$\ell\in W$, then $\ell\nleq qF$ for all $q\in Q$.  So suppose $\ell\in
W$ and that, by way of contradiction, $\ell\leq (q_{\ell'})F$.  Then
we have $\ell\leq q_{\ell'}F_{\ell'}$ for all $\ell'\in W$ by
\eqref{definelittlebox}. In particular, $\ell\leq q_{\ell}F_\ell$, a
contradiction to the choice of $F_\ell$.  It follows that $F$ has
the desired property, establishing the theorem.
\end{proof}

Combining Theorems~\ref{flowform} and~\ref{littleboxes}, one
easily deduces the following corollary.

\begin{Cor}\label{presentation-flowform-inev}
Let $M$ be a finite $X$-generated group mapping monoid with
distinguished $\R$-class $R$ and let \pv V be a pseudovariety of monoids.  Then
$M\in \pv A\malce(\pv G\ast \pv V)$ if and only if $\mathsf{RLM}(M)\in \pv A\malce (\pv G\ast \pv V)$ and there
are no $\pv V$-inevitable elements of the form $(\{r,s\},\{\{r,s\}\})$ with $r\H s$ and $r\neq s\in R$.
\end{Cor}
\begin{proof}  Suppose first that there are no $\pv V$-inevitable elements of the form $(\{r,s\},\{\{r,s\}\})$ with $r\H s $ and $r\neq s\in R$.  Then the flow provided in the Little Boxes Theorem satisfies the conditions of Theorem~\ref{flowform}.  Conversely, if there is a $\pv V$-inevitable element of the form $(\{r,s\},\{\{r,s\}\})$ with $r\H s$ and $r\neq s$, then no flow satisfying the conditions of Theorem~\ref{flowform} can exist and therefore $M\notin \pv A\malce (\pv G\ast \pv V)$ by Theorem~\ref{flowform}.
\end{proof}

\section{The Flow Monoid}
In this section, we provide the tools for constructing an
effective lower bound for complexity.  The lower bound will
be established in the next section.  Again fix the notation $L=\sp$.  We write $I$ for $I_L$.

\subsection{Loopable elements}
To describe our lower bound, we need the notion of an $n$-loopable element of a monoid, which is defined inductively.  First we need some definitions from~\cite{lowerbounds2}; complete details can be found in~\cite[Section 4.12]{qtheor}.  Denote by $\pv R$ the pseudovariety of $\R$-trivial semigroups and by $\pv {ER}$ the pseudovariety of semigroups whose idempotents generate an $\R$-trivial semigroup.  Stiffler proved $\pv{ER}=\pv R\ast \pv G$~\cite{Stiffler,Almeida:book,Eilenberg}.  In addition, we shall also require Stiffler's switching rule: $\pv G\ast \pv R\subseteq \pv R\ast \pv G$; see~\cite[Corollary 4.5.3]{qtheor}.

\begin{Def}[Type I]
A submonoid $S'$ of a monoid $S$ is said to be \emph{Type I} if, for any relational morphism $\p\colon S\to T$ with $T\in \pv A$, there exists a submonoid $T'\leq T$ so that $T'\in \pv{ER}$ and $S'\subseteq T'\pinv$.\index{Type I}
\end{Def}

A monoid is said to be \emph{absolute Type} I if it is a Type I submonoid of itself.
Absolute Type I monoids were effectively characterized by the first two authors, Margolis and Pin in~\cite{HMPR}, see~\cite[Theorem 4.12.19]{qtheor}.  It follows from a result of the authors~\cite{ourstablepairs} that it is decidable whether a submonoid of a monoid is Type I.  We briefly explain.

\begin{Def}[$\pv V$-stable pair]
Let $S$ be a monoid and suppose that $A\subseteq S$ and $S'$ is a submonoid of $S$.  Then $(A,S')$ is called a \emph{$\pv V$-stable pair} if, for all relational morphisms $\p\colon S\to T$ with $T\in \pv V$, there is an element $t\in T$ so that $A\subseteq t\pinv$ and $S'\subseteq \stab t\pinv$ where $\stab t=\{u\in T\mid tu=t\}$ is the right stabilizer of $t$ in $T$.\index{stable pair}
\end{Def}

The following theorem was proved by the authors in~\cite{ourstablepairs}.

\begin{Thm}
The set of $\pv A$-stable pairs of a finite monoid is effectively constructible.
\end{Thm}

The next proposition relates $\pv A$-stable pairs to Type I submonoids.  We shall make use of relatively free profinite monoids in the proof. Let $\pv V$ be a pseudovariety of monoids.
If $X$ is a finite set, then $\wh{F}_{\pv V}(X)$ denotes the free pro-$\pv V$ monoid on $X$ for $\pv V$ a pseudovariety of monoids~\cite[Chapter 3]{qtheor}.   We write $F_{\pv V}^{\omega}(X)$ for the unary submonoid of $\wh{F}_{\pv V}(X)$ generated by $X$ with $()^{\omega}$ as the unary operation.

\begin{Prop}\label{TypeI}
Let $S$ be a finite monoid and $S'$ a submonoid. Then the following are equivalent:
\begin{enumerate}
\item $S'$ is a Type I submonoid of $S$;
\item There exists $s\in S$ so that $(\{s\},S')$ is an $\pv A$-stable pair;
\item There exists $A\subseteq S$ so that $(A,S')$ is an $\pv A$-stable pair.
\end{enumerate}
Consequently, it is decidable whether a submonoid of $S$ is Type I.
\end{Prop}
\begin{proof}
Fix a generating set $X$ for $S$.  Let $\rho_{\pv A}\colon S\to \wh{F}_{\pv A}(X)$ be the canonical relational morphism: the graph of $\rho_{\pv A}$ is the closed submonoid of $S\times \wh{F}_{\pv A}(X)$ generated by all pairs $([x]_S,x)$ with $x\in X$.

To see that (1) implies (2),   we use~\cite[Corollary 3.7.5]{qtheor} to assert that there is a closed submonoid $T\leq \wh{F}_{\pv A}(X)$ that is pro-$\pv {ER}$ so that $S'\subseteq T\rho_{\pv A}\inv$.  Let $J$ be the minimal ideal of $T$; then $J$ is pro-$\pv A$ (and so has trivial maximal subgroups) and has a unique $\mathscr L$-class. It follows from stability of profinite semigroups that if $t\in J$ is any element, then $T\subseteq \stab t$.  So if $s\in t\rho_{\pv V}\inv$, then $(\{s\},S')$ is an $\pv A$-stable pair by~\cite[Theorem 2.6]{ourstablepairs}.  Clearly (2) implies (3).  For (3) implies (1), we note that~\cite[Theorem 2.6]{ourstablepairs} implies there exists $t\in \wh{F}_{\pv A}(X)$ so that $A\subseteq t\rho_{\pv A}\inv$ and $S'\subseteq \stab t\rho_{\pv A}\inv$.
But~\cite[Theorem 4.1]{ourstablepairs} shows that $\stab t$ is a chain in its own (internal) $\mathscr L$-order and hence it must be $\R$-trivial since it is pro-$\pv A$.  Thus $S'$ is a Type I submonoid of $S$ by another application of~\cite[Corollary 3.7.5]{qtheor}.

The decidability result is immediate from the decidability of (2) or (3).
\end{proof}

We also need the notion of Type II elements.
\begin{Def}[Type II]
An element $s$ of a monoid $S$ is said to be of \emph{Type II} if, for all relational morphisms $\p\colon S\to G$ with $G\in \pv G$, one has $s\in 1\pinv$.  Denote by $\KG S$ the set of all Type II elements of $S$; it is a submonoid.\index{Type II}
\end{Def}

If $a,b\in S$ are such that $aba=a$ and $s\in S$, then we say $asb,bsa$ are \emph{weak conjugates} of $s$.  The following effective characterization of Type II elements was conjectured by the second author and proved by Ash~\cite{Ash}, and independently Ribes and Zalesskii~\cite{RZ}; see~\cite[Theorem 4.17.30]{qtheor} or~\cite{KarlTypeII} for perhaps the easiest proofs.

\begin{Thm}
Let $S$ be a monoid.  Then $\KG S$ is the least submonoid of $S$ closed under weak conjugation.
\end{Thm}

We now wish to define the notion of a $\pv V$-aperiodic element of a monoid, where $\pv V$ is a pseudovariety of monoids.

\begin{Def}[$\pv V$-aperiodic element]
An element $s\in S$ of a profinite monoid is called \emph{aperiodic} if $s^{\omega}=s^{\omega+1}$.  If $\pv V$ is a pseudovariety of monoids, then an element $s$ of a monoid $S$ is called \emph{$\pv V$-aperiodic} if, for all relational morphisms $\p\colon S\to T$ with $T\in \pv V$, there exists an aperiodic element $t\in T$ so that $t\in s\p$.\index{$\pv V$-aperiodic element}
\end{Def}

\begin{Rmk}\label{obvious}
Notice that any element is $\pv A$-aperiodic.   An element of a monoid is $\pv G$-aperiodic if and only if it is of Type II.  It is obvious that if $S'\leq S$ and $s$ is $\pv V$-aperiodic in $S'$, then $s$ is $\pv V$-aperiodic in $S$.  Also note that if $S\in \pv V$, then each $\pv V$-aperiodic element of $S$ must be, in fact, aperiodic (consider the identity homomorphism).
\end{Rmk}

Recall from~\cite[Definition 3.6.25]{qtheor} that a subset $A$ of a finite monoid $S$ is said to
be \emph{\pv V-like}\index{$\pv V$-like} with respect to a pseudovariety $\pv W$ if, for all relational morphisms $\p\colon S\to W$ with
$W\in \pv W$, there exists a submonoid $V\leq W$ so that $V\in \pv V$ and $A\subseteq V\pinv$.  It follows immediately from the definitions that a submonoid $S'$ of $S$ is Type I if and only if it is $\pv{ER}$-like with respect to $\pv V$ and that an element $s\in S$ is $\pv V$-aperiodic if and only if $\{s\}$ is $\pv A$-like with respect to $\pv V$.  The following is then a special case of what is proved in~\cite[page 179]{qtheor}.

\begin{Prop}\label{witnessaperiodic}
Let $S$ be a finite monoid and $\pv V$ a pseudovariety.  Then there exists a relational morphism $\p\colon S\to V$ with $V\in \pv V$ such that $s\in S$ is $\pv V$-aperiodic if and only if there exists an aperiodic element $v\in V$ so that $v\in s\p$.
\end{Prop}

We aim to provide a computable set of $\pv C_n$-aperiodic elements.  We begin with a straightforward reduction to a generating set for the pseudovariety.

\begin{Lemma}\label{sillylemma}
Let $\p\colon S\to T$ be a relational morphism and $d\colon T\to U$ a division.  Fix $s\in S$ and suppose there is an aperiodic element $u\in U$ with $u\in s\p d$.  Then there is an aperiodic element $t\in T$ so that $t\in s\p$.
\end{Lemma}
\begin{proof}
Choose $t\in T$ so that $t\in s\p$ and $u\in td$.  We claim $t$ is aperiodic.  Indeed, $u{^\omega}\in t^{\omega}d\cap t^{\omega+1}d$ and so $t^{\omega}=t^{\omega+1}$ as $d$ is a division.
\end{proof}

The following lemma shows how to generate $\pv V\ast \pv G\ast \pv A$-aperiodic elements.

\begin{Lemma}\label{generateap}
Let $\pv V$ be a pseudovariety such that $\pv V\ast \pv R=\pv V$.  Let $S'$ be a Type I submonoid of a monoid $S$ and suppose that $s$ is a $\pv V$-aperiodic element of $\KG {S'}$.  Then $s$ is $\pv V\ast \pv G\ast \pv A$-aperiodic in $S$.
\end{Lemma}
\begin{proof}
By Lemma~\ref{sillylemma} it suffices to show that if $\p\colon S\to B\rtimes A$ is a relational morphism with $B\in \pv V\ast \pv G$ and $A\in \pv A$, then there exists an aperiodic element $t\in B\rtimes A$ so that $s\in t\pinv$.

Let $\pi\colon B\rtimes A\to A$ be the semidirect product projection.  Setting $\psi = \p\pi$, we can find a submonoid $A'\leq A$ with $A'\in \pv {ER}$ and $S'\subseteq A'\psi\inv$.  Hence, by restricting $\p$, we can obtain a relational morphism $\eta\colon S'\to B\rtimes A'$ with $\eta\subseteq \p$.  Set $T=B\rtimes A'$.  Then $T\in \pv V\ast \pv G\ast\pv {ER}$ and
\[\pv V\ast \pv G\ast \pv {ER}=\pv V\ast \pv G\ast \pv R\ast \pv G\subseteq \pv V\ast \pv R\ast \pv G\ast \pv G\subseteq \pv V\ast \pv G\] and so $\KG T\in \pv V$.  Choose a finite monoid $R$ so that $\eta=\alpha\inv\beta$ with $\alpha\colon R\twoheadrightarrow S'$ an onto homomorphism and $\beta\colon R\to T$ a homomorphism.  By~\cite[Proposition 4.12.6]{qtheor}, we have that $\KG R\alpha=\KG {S'}$  and $\KG R\beta \subseteq \KG T$.  Hence, we obtain a relational morphism $\tau\colon \KG {S'}\to \KG{T}$ with $\tau\subseteq \eta$.  Since $\KG T\in \pv V$ and $s$ is $\pv V$-aperiodic, there exists an aperiodic element $t\in \KG T\subseteq T$ so that $s\in t\tau\inv\subseteq t\eta\inv\subseteq t\pinv$.  This completes the proof that $s$ is $\pv V\ast \pv G\ast \pv A$-aperiodic in $S$.
\end{proof}

With Lemma~\ref{generateap} in hand, we can define recursively a class of $\pv C_n$-aperiodic elements, which we call $n$-loopable elements.

\begin{Def}[$n$-loopable element]
An \emph{$n$-loopable} element of a monoid $S$ is defined recursively as follows:
\begin{itemize}
\item Any element of $S$ is $0$-loopable;
\item An element $s\in S$ is $n$-loopable in $S$, for $n\geq 1$, if there exists a Type I submonoid $T$ of $S$ so that $s$ is an $(n-1)$-loopable element of $\KG T$.
\end{itemize}\index{$n$-loopable element}
\end{Def}

Since one can effectively find all the  Type I submonoids of a monoid and the Type II submonoid is effectively computable, it follows that the set of $n$-loopable elements of a monoid is effectively computable.
An easy induction establishes that $n$-loopable elements are $\pv C_n$-aperiodic.

\begin{Prop}\label{loopableisaperiodic}
Suppose that $s\in S$ is $n$-loopable.  Then $s$ is $\pv C_n$-aperiodic.
\end{Prop}
\begin{proof}
We proceed by induction on $n$, the case $n=0$ being trivial since all elements are $\pv A$-aperiodic.  Suppose the result is true for $n\geq 0$ and suppose $s\in S$ is $(n+1)$-loopable.  Then we can find a Type I submonoid $T$ of $S$ so that $s$ is $n$-loopable in $\KG T$.  By induction, $s$ is $\pv C_n$-aperiodic in $\KG T$ and so is $\pv C_n\ast \pv G\ast \pv A$-aperiodic in $T$ by Lemma~\ref{generateap}.  Thus $s$ is $\pv C_{n+1}$-aperiodic in $S$ by Remark~\ref{obvious}.
\end{proof}

Let us show that computing $\pv C_n$-aperiodic elements is tantamount to computing complexity.
\begin{Prop}\label{approp1}
Let $\pv V$ be a pseudovariety of monoids.  Then $s\in S$ is $\pv V$-aperiodic if and only if it is $\pv A\malce \pv V$-aperiodic.
\end{Prop}
\begin{proof}
Clearly any $\pv A\malce \pv V$-aperiodic element is $\pv V$-aperiodic.  For the converse, by Lemma~\ref{sillylemma}, it suffices to consider relational morphisms $\p\colon S\to T$ so that $T$ admits an aperiodic homomorphism $\psi\colon T\to V$ with $V\in \pv V$.  Let $v\in V$ be aperiodic with $v\in s\p\psi$.  By Lemma~\cite[Lemma 4.4.4]{qtheor}, $\langle v\rangle\psi\inv$ is aperiodic.  Hence if $t\in T$ is such that $t\in s\p$ and $v=t\psi$, then $t$ is aperiodic.
\end{proof}

Consequently, we have the following membership criterion for $\pv A\malce \pv V$.
\begin{Cor}
Let $\pv V$ be a pseudovariety.  Then the following are equivalent:
\begin{enumerate}
\item $S\in \pv A\malce \pv V$;
\item Each $\pv A\malce \pv V$-aperiodic element of $S$ is aperiodic;
\item Each $\pv V$-aperiodic element of $S$ is aperiodic.
\end{enumerate}
\end{Cor}
\begin{proof}
The implication from (1) to (2) was mentioned in Remark~\ref{obvious}.  Trivially, (2) implies (3).  Suppose (3) holds.  Find a relational morphism $\p\colon S\to V$ with $V\in \pv V$ as per Proposition~\ref{witnessaperiodic}.  Let $e\in V$ be an idempotent.  Then since $e$ is aperiodic, each element of $e\pinv$ is $\pv V$-aperiodic and hence aperiodic by assumption.  Thus $S\in \pv A\malce \pv V$.
\end{proof}

The Fundamental Lemma of Complexity~\cite{flc,TilsonXII,qtheor} asserts that $\pv C_n=\pv A\malce (\pv G\ast \pv C_{n-1})$, so we obtain the following consequence.

\begin{Cor}\label{complexityasaperiodic}
A monoid $S$ belongs to $\pv C_n$ if and only if all its $\pv C_n$-aperiodic elements are aperiodic.  Hence the complexity of $S$ is the least $n$ so that all $\pv C_n$-aperiodic elements of $S$ are aperiodic.
\end{Cor}

The next corollary is a rephrasing of the Type I-Type II lower bound of~\cite{lowerbounds2} in the language of loopable elements.
\begin{Cor}
Let $n$ be the least non-negative integer so that each $n$-loopable element of $S$ is aperiodic.  Then $S$ has complexity at least $n$.
\end{Cor}

The Tall Fork from~\cite[Section 4.14]{qtheor} has the property that all its $1$-loopable elements are aperiodic, but it has complexity $2$.  Thus the $n$-loopable elements form a proper subset of the $\pv C_n$-aperiodic elements.

\subsection{Values and inevitable flows}

Fix again an $X$-generated group mapping monoid $M$ with distinguished $\R$-class $R$ and put $L=\sp$.
In this subsection we single out a certain submonoid of the abstract flow monoid consisting of so called $\pv V$-inevitable elements.

If $(Q,X)$ is a partial automaton with transition monoid in $\pv V$, then $\widehat F_{\pv V}(X)$ acts
naturally on $Q$.  In what follows, if $F$ is a complete flow on a partial automaton $\Ac$ and $qt$ is undefined for $q\in Q$, then we interpret $qtF$ as $\square F=B$.

\begin{Def}[$\pv V$-inevitable flow]\label{definevflow}
Let $\pv V$ be a pseudovariety of monoids.  An abstract flow $f\in \mathscr C(L^2)$ is said to be \emph{$\pv V$-inevitable} if, for all complete flows $F$ on a partial automaton $\Ac=(Q,X)$ with
transition monoid $N\in \pv V$, there exists $t\in N$ so that one has
\begin{equation}\label{inevflow}
qF\xrightarrow {\ov f} qtF
\end{equation}
for all states $q\in Q$.\index{$\pv V$-inevitable flow}
\end{Def}

It turns out that, for any complete flow on an automaton with transition monoid in $\pv V$, the values of the flow are stable under back-flow by any $\pv V$-inevitable abstract flow.

\begin{Lemma}\label{realsflowsfclosed}
Let $F$ be a complete flow on a partial automaton $\mathscr A=(Q,A)$ with transition monoid $N\in \pv V$.  Suppose that $f$ is a $\pv V$-inevitable abstract flow and $q\in Q$.  Then $(qF)\overleftarrow{f}=qF$.
\end{Lemma}
\begin{proof}
Choose $t\in N$ so that \eqref{inevflow} holds.  Then evidently $qF\in \dom {\ov f}$ and hence stable for $\overleftarrow{f}$.
\end{proof}

\begin{Prop}\label{fixB}
Let $f$ be a $\pv V$-inevitable abstract flow.  Then $B\xrightarrow{\ov f} B$.
\end{Prop}
\begin{proof}
Consider the  one-state complete automaton over $X$ \[\Ac = q_0\looop {X}\] and define a complete flow $F$ by $q_0F=B$.  Then by the definition of $\pv V$-inevitability $B=q_0F \xrightarrow{\ov f} q_0F=B$ since the transition monoid is trivial.
\end{proof}

It turns out that one can switch the order of the quantifiers in Definition~\ref{definevflow}.

\begin{Prop}\label{quantifierswitch}
An abstract flow $f\in \mathscr C(L^2)$ is $\pv V$-inevitable if and only if, for all partial automata $\Ac=(Q,X)$ with transition monoid $N\in \pv V$, there exists $t\in N$ so that, for all complete flows $F$ on $\mathscr A$, one has
\begin{equation*}
qF\xrightarrow {\ov f} qtF
\end{equation*}
for all states $q\in Q$.
\end{Prop}
\begin{proof}
Trivially, if the condition given in the proposition is verified, then $f$ is $\pv V$-inevitable.  For the converse, assume $f$ is $\pv V$-inevitable and suppose $\mathscr A=(Q,X)$ is a partial automaton with transition monoid $N\in \pv V$.  Let $\{F_1,\ldots,F_m\}$ be the \emph{finite} set of all complete flows on $\mathscr A$. Consider the partial automaton $\mathscr A'=(Q\times \{1,\ldots,m\},X)$ whose transitions, for $x\in X$, are of the form $(q,i)x=(qx,i)$ if $qx$ is defined and undefined otherwise.  In other words $\mathscr A'$ is a disjoint union of $m$ copies of $\mathscr A$.  Evidently, the transition monoid of $\mathscr A'$ is $N$. Define a flow $F$ on $\mathscr A'$ by $(q,i)F= qF_i$.  It is straightforward to verify that $F$ is a complete flow.  Since $f$ is $\pv V$-inevitable, there exists $t\in N$ so that \[qF_i=(q,i)F\xrightarrow{\ov f}(q,i)tF=qtF_i\] for all $q\in Q$ and $i=1,\ldots, m$.  This completes the proof.
\end{proof}

As with many inevitability notions, there is always a finite model witnessing $\pv V$-inevitable flows.

\begin{Prop}\label{inevflowcompactness}
There exists a complete flow $F$ on a finite partial automaton $\mathscr A$ with transition monoid $N\in \pv V$ so that $f\in \mathscr C(L^2)$ is $\pv V$-inevitable if and only if there exists $t\in T$ so that \eqref{inevflow} holds.
\end{Prop}
\begin{proof}
Let $A$ be the set of elements of $\mathscr C(L^2)$ that are not $\pv V$-inevitable; it is a finite set.  For each $g\in A$, choose a complete flow $F_g$ on a partial automaton $\mathscr A_g=(Q_g,X)$ with transition monoid $N_g\in \pv V$ so that, for all $t\in N_g$, there is a state $q_{g,t}\in Q_g$ so that $(q_{g,t}F_g,q_{g,t}tF_g)$ is not stable for $g$.  Let $\mathscr A=\coprod_{g\in A} \mathscr A_g$ be the disjoint union of these automata and let $F$ be the flow defined on $\mathscr A$ by $F|_{\mathscr A_g}=F_g$.  Then the transition monoid $N$ of $\mathscr A$ is a subdirect product of the $N_g$ and so belongs to $\pv V$.  Clearly, $F$ is a complete flow on $\mathscr A$.   Let $f\in A$ and suppose there exists $t\in N$ so that \eqref{inevflow} holds for all states $q$ of $\mathscr A$.  Then in particular, \[q_{f,t}F_f=q_{f,t}F\xrightarrow{\ov f}q_{f,t}tF=q_{f,t}tF_f\] a contradiction.  This shows that $F$ is the desired complete flow.
\end{proof}

Next we introduce the notion of a value for an abstract flow; it will turn out that an abstract flow has a value if and only if it is $\pv V$-inevitable.

\begin{Def}[Values]
Let \pv V be a pseudovariety of monoids.  An element $t\in \wh F_{\pv V}(X)$
is said to be a \emph{value} of $f\in \mathscr C(L^2)$ (relative to $\pv V$) if,
for all complete flows $F$ on a partial automaton $\Ac=(Q,X)$ with
transition monoid in \pv V and all states $q\in Q$, we have
\begin{equation}\label{valueflow}
qF\xrightarrow {\ov f} qtF.
\end{equation}
We use $f\upsilon_{\pv V}$ to denote the set of values of $f$.\index{value}
\end{Def}

Values in the aperiodic setting are very closely related to what are
called bases in geometric semigroup theory~\cite{gst}.
A standard compactness argument shows that the $\pv V$-inevitable flows are exactly those which admit values.

\begin{Prop}\label{value=inev}
Let $\pv V$ be a pseudovariety of monoids.  Then $f\in \mathscr C(L^2)$ is $\pv V$-inevitable if and only if it has a value in $\wh{F}_{\pv V}(X)$.
\end{Prop}
\begin{proof}
Suppose first that $t\in f\upsilon_{\pv V}$ and let $F$ be a complete flow on an automaton $\mathscr A=(Q,X)$ with transition monoid $N\in \pv V$.  Then \eqref{inevflow} holds since \eqref{valueflow} does.

Conversely, suppose $f$ is $\pv V$-inevitable.  Write $\wh{F}_{\pv V}(X)=\ilim_{i\in A} V_i$ where $\{V_i\mid i\in A\}$ is the set of all $X$-generated monoids in $\pv V$.  Consider the Cayley graph $\mathscr V_i =(V_i,X)$ and let $C_i$ be the set of all $t\in V_i$ so that \eqref{inevflow} holds for all complete flows $F$ on $\mathscr V_i$; by Proposition~\ref{quantifierswitch}, $C_i\neq \emptyset$.  We claim that $\{C_i\mid i\in A\}$ is an inverse subsystem of $\{V_i\mid i\in A\}$.  Indeed, suppose that $i\geq j$ and $\pi_{ij}\colon V_i\to V_j$ is the projection.  Let $F$ be a complete flow on $\mathscr V_j$.  Then $\pi_{ij}F$ is a complete flow on $\mathscr V_i$ since $q\pi_{ij}F\xrightarrow{\ov x}q\pi_{ij}xF = qx\pi_{ij}F$.  Suppose now that $t\in C_i$.  Let $p\in V_j$ and choose a preimage $q\in p\pi_{ij}\inv$.  Then \[pF = q\pi_{ij}F\xrightarrow{\ov f} (qt)\pi_{ij}F = (p(t\pi_{ij}))F\] and so $t\pi_{ij}\in C_j$.

Since an inverse limit of  non-empty finite sets is non-empty~\cite[Lemma 3.1.22]{qtheor} $C=\ilim_{i\in A}C_i\leq \wh {F}_{\pv V}(X)$ is non-empty.  Choose $t\in C$; we claim that $t$ is a value for $f$.  Let $\mathscr A=(Q,X)$ be any partial automaton with transition monoid $V_i$ in $\pv V$ and let $F$ be a complete flow on $\mathscr A$.  Let $q\in Q$.  We need to show that $qF\xrightarrow{\ov f}qtF$.  First we view $F$ as a flow on the complete automaton $\mathscr A^{\square} = (Q\cup \{\square\},X)$ by defining $\square F=B$; note that $V_i$ is still the transition monoid of $\mathscr A^{\square}$.  Next define a complete flow $F'$ on $\mathscr V_i$ by putting $vF' = qvF$.  To verify that this is a complete flow, note that \[vF'=qvF\xrightarrow{\ov x}qvxF = vxF'\] since $F$ is a complete flow.  Let $\pi_i\colon \wh{F}_{\pv V}(X)\to V_i$ be the continuous projection.  Then $C\pi_i\subseteq C_i$ and so if $I$ is the identity of $V$, then  \[qF=IF'\xrightarrow{\ov f}It\pi_iF'=qtF,\] as required.  This shows that $C\subseteq f\upsilon_{\pv V}$, completing the proof.
\end{proof}

\subsection{The values lemma}
Here we establish that the $\pv V$-inevitable flows form a submonoid of $\mathscr C(L^2)$ with certain closure properties.

\begin{Lemma}[Values lemma]\label{valueslemma}
Let \pv V be a pseudovariety.  The collection $\mathsf M_{\pv V}(L)$ of all elements of
$\mathscr C(L^2)$ that have values satisfies:
\begin{enumerate}
\item  {\rm (Identity)}\ The multiplicative identity $I$ of $\mathscr
C(L^2)$ is in $\mathsf M_{\pv V}(L)$;
\item {\rm (Points)}\ For all $x\in X$,
free flow along $x$ belongs to $\mathsf M_{\pv V}(L)$;
\item {\rm (Products)}\
If $f_1,f_2\in \mathsf M_{\pv V}(L)$, then $f_1f_2\in \mathsf M_{\pv V}(L)$;
\item  {\rm (Vacuum)}\ If $f\in \mathsf M_{\pv V}(L)$, then $\overleftarrow f\in \mathsf M_{\pv V}(L)$;
\item  {\rm (Aperiodic Action)}\ If $f\in \mathsf M_{\pv V}(L)$ is a $\pv V$-aperiodic element,
then $f^{\omega+*}\in \mathsf M_{\pv V}(L)$;
\item {\rm (Pointlikes)}\  If $A\subseteq \mathsf M_{\pv V}(L)$ is $\pv V$-pointlike, then $\bigvee A\in \mathsf M_{\pv V}(L)$;
\item {\rm (Stable pairs)}\ If $(A,T)$ is a $\pv V$-stable pair for $\mathsf M_{\pv V}(L)$, then we have $\bigvee A\cdot \left(\bigvee T\right)^*\in M_{\pv V}(L)$.
\end{enumerate}
Moreover, $\upsilon_{\pv V}\colon  \mathsf M_{\pv V}(L)\to \wh F_{\pv V}(X)$ is a relational morphism of
profinite monoids.
\end{Lemma}
\begin{proof}
First we verify (Identity).   We do this by showing that
the empty word $\varepsilon$ is a value for $I$.  Indeed, the stable
pairs of $I$ are just the elements of the form $(\ell,\ell)$ with
$\ell\in L$.  So if $\Ac=(Q,X)$ is any partial automaton, then
$qF\xrightarrow {\ov I} qF$ for all complete flows $F$ on $\Ac$.

Next we verify (Points).  We claim that $x$ is a value for
free flow along $x$.  Indeed, by definition of a complete flow on a
partial automaton $\Ac=(Q,X)$, if $q\in Q$, then \[qF\xrightarrow {\ov x}
qxF,\] establishing that $x$ is a value for free flow along $x$.

Turning to (Products), we claim that if $u\in f\upsilon_{\pv V}$ and $t\in g\upsilon_{\pv V}$, then $ut\in
(fg)\upsilon_{\pv V}$. Indeed, let $F$ be a complete flow on a
partial automaton $\Ac=(Q,X)$ with transition monoid in \pv V. Then,
for all $q\in Q$,
\[qF\xrightarrow {\ov f} quF\xrightarrow {\ov g}qutF.\] From this we obtain
$qF\xrightarrow {\ov {fg}} qutF$, as required.

The proofs for (Identity) and (Products) shows that $\upsilon_{\pv V}$ is
indeed a relational morphism of monoids.  But if $t\in
\ov{f\upsilon_{\pv V}}$, then in every finite $X$-generated
monoid in $\pv V$, $t$ agrees with a value of $f$ and so \eqref{valueflow} holds.  Consequently, $t$ is a value for $f$. Thus
$f\upsilon_{\pv V}$ is closed and so $\upsilon_{\pv V}$ is
a relational morphism of profinite monoids.

Next, we verify (Vacuum).  Suppose $f\in \mathsf M_{\pv V}(L)$; we claim that the empty word $\varepsilon$ is a value for $\overleftarrow{f}$.  Let $F$ be a complete flow on a partial automaton with transition monoid in $\pv V$.  By Proposition~\ref{value=inev}, $f$ is $\pv V$-inevitable and hence, by Lemma~\ref{realsflowsfclosed}, $(qF)\overleftarrow{f}=qF$. Therefore, $(qF,qF)$ is a stable pair for $\overleftarrow{f}$, establishing that $\varepsilon\in \overleftarrow{f}\upsilon_{\pv V}$.

To check (Aperiodic Action), suppose that $f\in \mathsf M_{\pv V}(L)$ is $\pv V$-aperiodic; we show that $f^{\omst}$ is $\pv V$-inevitable and hence has a value.  Let $\Ac = (Q,X)$
be a partial automaton with transition monoid $N\in \pv V$.   Let $\pi\colon \wh F_{\pv V}(X)\to N$ be the canonical projection.  Then $\upsilon_{\pv V}\pi\colon \mathsf M_{\pv V}(L)\to N$ is a relational morphism and so there exists an aperiodic element $n\in N$ and $t\in f\upsilon_{\pv V}$ so that $t\pi =n$.

Let $F$ be a complete flow on $\Ac$. We claim that, for all states $q\in Q$, we have $qF\xrightarrow{\ov{f^{\omst}}}(qn^{\omega})F$.   First note that $t^{\omega}$ is a value for $f^{\omega}$ since $\upsilon_{\pv V}$ is a relational morphism of profinite monoids. Thus $qF\xrightarrow{\ov {f^{\omega}}} qt^{\omega}F$.  Moreover, \[qn^{\omega}F\xrightarrow{\ov f}qn^{\omega}nF=qn^{\omega}F\] since $t$ is a value for $f$, $n=t\pi$ and $n$ is aperiodic.  Thus \[qF\plstar{\ov{f^{\omega}}}{\ov f}qn^{\omega}F;\] in other words $qF\xrightarrow{\ov{f^{\omst}}}qn^{\omega}F$.  This establishes that $f$ is $\pv V$-inevitable.

Next suppose that $A\subseteq \mathsf M_{\pv V}(L)$ is $\pv V$-pointlike.  Let $F$ be a complete flow on a partial automaton $\Ac = (Q,X)$ with transition monoid $N\in \pv V$.  Let $\pi\colon \wh F_{\pv V}(X)\to N$ be the canonical projection.  Then $\psi=\upsilon_{\pv V}\pi\colon \mathsf M_{\pv V}(L)\to N$ is a relational morphism and so there exists an  element $n\in N$ so that $A\subseteq n\psi\inv$.  Hence, for each $a\in A$, we can find $t_a\in a\upsilon_{\pv V}$ so that $t_a\pi =n$.  Then $qF\xrightarrow{\ov a}qnF$ is stable for all $a\in A$ and therefore \[qF\xrightarrow{\ov {\bigvee A}}qnF.\]  Thus $\bigvee A$ is $\pv V$-inevitable and so belongs to $\mathsf M_{\pv V}(L)$.

Finally, suppose that $(A,T)$ is a $\pv V$-stable pair for $\mathsf M_{\pv V}(L)$.  As before,  let $F$ be a complete flow on a partial automaton $\Ac = (Q,X)$ with transition monoid $N\in \pv V$ and denote by $\pi\colon \wh F_{\pv V}(X)\to N$ the canonical projection.  Then $\psi=\upsilon_{\pv V}\pi\colon \mathsf M_{\pv V}(L)\to N$ is a relational morphism and so there exists an  element $n\in N$ so that $A\subseteq n\psi\inv$ and $T\subseteq \stab n\psi\inv$.  It follows that \[qF\plstar{\ov a}{\ov f}qnF\] for all $a\in A$ and $f\in T$ and so \[qF\xrightarrow{\ov{\bigvee A\cdot\left(\bigvee T\right)^*}}qnF.\]  We conclude that $\bigvee A\cdot\left(\bigvee T\right)^*\in \mathsf M_{\pv V}(L)$.
\end{proof}

\subsection{The $n$-flow monoid}
To create our lower bound for complexity $n$, we would like to use $\mathsf M_{\pv C_{n-1}}(L)$, but it is not clear that this set is computable.  In fact, computability of $\pv C_n$-aperiodic elements for all $n\geq 0$ implies the decidability of complexity by Corollary~\ref{complexityasaperiodic}.  So instead we use the effectively constructible set of $(n-1)$-loopable elements.

\begin{Def}[$n$-flow monoid]\label{flowmonoid}
The $n$-flow monoid $\mathsf M_n(L)$, for $n\geq 0$, is the smallest
subset of $\mathscr C (L^2)$ satisfying the following axioms:
\begin{enumerate}
\item  {\rm (Identity)}\ The multiplicative identity $I$ of $\mathscr
C(L^2)$ is in $\mathsf M_n(L)$;
\item {\rm (Points)}\ For all $x\in X$,
free flow along $x$ belongs to $\mathsf M_n(L)$;
\item {\rm (Products)}\
If $f_1,f_2\in \mathsf M_n(L)$, then $f_1f_2\in \mathsf M_n(L)$;
\item  {\rm (Vacuum)}\ If $f\in \mathsf M_n(L)$, then $\overleftarrow f\in \mathsf M_n(L)$;
\item  {\rm (Loops)}\ If $f\in \mathsf M_n(L)$ is $n$-loopable,
then $f^{\omega+*}\in \mathsf M_n(L).$
\end{enumerate}\index{$n$-flow monoid}
\end{Def}

We obtain as an immediate corollary of Lemma~\ref{valueslemma} and Proposition~\ref{loopableisaperiodic}:

\begin{Cor}\label{usablevalues}
The set $\mathsf M_n(L)$ is an effectively computable submonoid of $\mathsf M_{\pv C_n}(L)$.  Hence each element of $\mathsf M_n(L)$ has a value relative to $\pv C_n$ and $\upsilon_{\pv
C_n}\colon M_n(L)\rightarrow \wh{F}_{\pv C_n}(X)$ is a relational morphism.  Consequently, each element of $M_n(L)$ is $\pv C_n$-inevitable.
\end{Cor}

The proof in fact shows the following: each element of $\mathsf
M_0(L)$ has a value in $F^{\omega}_{\pv A}(X)$.  Thus for complexity one, we are looking at some type of tameness as per~\cite{Almeida&Steinberg:2000,Almeida&Steinberg:1998}.

It is natural to ask why we do not choose some of the other properties from Lemma~\ref{valueslemma} in our definition of $M_n(L)$.  For instance, $\pv A$-pointlikes are decidable~\cite{Henckell,ouraperioidcpointlikes,qtheor}, so why not allow them to be joined in the definition of $M_0(L)$?  It turns out that they are not necessary.

\begin{Prop}
Let $A$ be an $\pv A$-pointlike subset of $M_0(L)$.  Then there exists $f\in M_0(L)$ so that $\bigvee A\leq f$.
\end{Prop}
\begin{proof}
According to~\cite[Theorem 4.19.2]{qtheor}, the collection of $\pv A$-pointlike subsets of $M_0(L)$ is the smallest subsemigroup of $P(M_0(L))$ containing the singletons and closed under $Z\mapsto Z^{\omega}\bigcup_{k\geq 0}Z^k$.  Let $B$ be the collection of subsets of $M_0(L)$ with an upper bound in $M_0(L)$.  We show that $B$ satisfies these properties.  Clearly all singletons belong to $B$.  if $Y,Z\in B$ and $f,g\in M_0(L)$ bound $Y$ and $Z$ respectively, then $fg$ bounds each element of $YZ$ since $M_0(L)$ is an ordered monoid.  Finally, suppose $f\in M_0(L)$ is an upper bound for $Z\in B$.  Then each element of $Z^k$ is below $f^k$, which in turn is below $f^*$.  Thus $f^{\omst}$ is an upper bound for $Z^{\omega}\bigcup_{k\geq 0}Z^k$.  This completes the proof.
\end{proof}

\subsection{The $\Fc$-operator}
We continue to denote $\sp$ by $L$.  Fix $n\geq 0$.
The following definition is crucial to what follows.

\begin{Def}[$\Fc$-operator]
Define a closure operator
$\Fc\in \mathscr C(L)$ by
\begin{equation}\label{Fopeq}
\Fc = \bigvee _{f\in \mathsf M_n(L)} \overleftarrow {f}.
\end{equation}
We sometimes call $\Fc$ the \emph{vacuum}.\index{$\Fc$-operator}
\end{Def}

The stable
set for $\Fc$ is the intersection $\bigcap_{f\in \mathsf M_n(L)} \dom \ov f$.

\begin{Prop}
Viewing $\Fc$ as a two-variable closure operator, we have $\Fc$ is an idempotent of $\mathsf M_n(L)$.
\end{Prop}
\begin{proof}
This is immediate from Axiom (Products), Axiom (Vacuum) and Proposition~\ref{twoasone}.
\end{proof}

Thus it is natural to localize $\mathsf M_n(L)$ at $\Fc$, i.e., work with the subsemigroup $\Fc \mathsf M_n(L)\Fc$ of $\mathsf M_n(L)$.

\begin{Def}[\Fc-stable]
An element $\ell\in L$ is said to be \emph{$\Fc$-stable} if $\ell\Fc
= \ell$. An element $f$ of $\mathsf M_n(L)$ is said to be
\emph{\Fc-stable} if $\Fc \FFF f = f$.\index{$\Fc$-stable}
\end{Def}

Our first example of an $\Fc$-stable subset is the bottom.

\begin{Prop}\label{FpreservesB}
The bottom $B$ of $L$ is $\Fc$-stable.
\end{Prop}
\begin{proof}
Proposition~\ref{fixB} implies $B=B\overleftarrow{\Fc}=B\Fc$.
\end{proof}

Our next observation is that any flow on an automaton of complexity at most $n$ takes on only $\Fc$-stable values.  More precisely, we have:

\begin{Lemma}\label{aperiodicvaluesareFstable}
Let $F$ be a complete flow on a partial automaton $\mathscr A=(Q,X)$ with transition monoid in $\pv C_n$.  Then $qF$ is $\Fc$-stable for all $q\in Q$.
\end{Lemma}
\begin{proof}
Since each element of $\mathsf M_n(L)$ is $\pv C_n$-inevitable by Corollary~\ref{usablevalues} and $\Fc = \overleftarrow{\Fc}$ (cf.~Remark~\ref{oneastwonoflow}), it follows that $(qF)\Fc = qF$ by Lemma~\ref{realsflowsfclosed}.
\end{proof}

As a consequence, we can prove that sets do not change under $\Fc$ and hence the points of $\sp$ are $\Fc$-stable.
\begin{Prop}\label{pointsareclosed}
If $Y\subseteq R$, then $(Y,\{Y\})$ is $\Fc$-stable.  In particular, if $r\in R$, then $(r,r)$ is $\Fc$-stable.
\end{Prop}
\begin{proof}
Let $\mathscr A$ be the complete automaton 
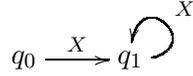
\[\xymatrix{q_0\ar[r]^X&q_1\ar@(r,u)_X}\] where an edge labelled $X$ is shorthand for a set of edges labelled by each element of $X$.   Consider the complete flow $F$ on $\mathscr A$ given by $q_0F=(Y,\{Y\})$ and $q_1F=(R,\{R\})$.  This is a flow since, for any $x\in X$, trivially $Yx\subseteq R$ and any partial map $Y/\{Y\}\to R/\{R\}$ is injective.  Of course, $(R,\{R\})$ is stable for $x^*$ all $x\in X$ being the top of $\sp$.  Since $\mathscr A$ has aperiodic transition monoid, it follows $(Y,\{Y\})\Fc = (Y,\{Y\})$ by Lemma~\ref{aperiodicvaluesareFstable}.
\end{proof}

Thus the vacuum only changes partitions and not sets.

\begin{Cor}
If $(Y,P)\in L$, then $(Y,P)\Fc = (Y,P')$ for some partition $P'$ on $Y$.
\end{Cor}
\begin{proof}
Since $(Y,P)\leq (Y,\{Y\})$, we have \[(Y,P)\leq (Y,P)\Fc \leq (Y,\{Y\})\Fc = (Y,\{Y\})\] by Proposition~\ref{pointsareclosed}.  Thus $(Y,P)\Fc=(Y,P')$ for some partition $P'$.
\end{proof}

The operator $\Fc$ captures the back-flow from all
elements of $\mathsf M_n(L)$.  More precisely we have the following
proposition.

\begin{Prop}\label{Fnobackflow}
Suppose $\ell\in L$ is \Fc-stable and let $f\in \mathsf M_n(L)$.
Then $\ell\overleftarrow {f} = \ell$, that is, $\ell\xrightarrow f B =
(\ell,\ell\overrightarrow f)$.
\end{Prop}
\begin{proof}
This follows since $\overleftarrow f\leq \Fc$ and hence the
fact that $\ell$ is \Fc-stable implies that $\ell$ is stable for
$\overleftarrow f$.
\end{proof}

\begin{Prop}\label{FMF}
The map $\mathscr C(L^2)\to \mathscr C (L^2)$ given by
$f\mapsto \Fc f\Fc$ is a closure operator satisfying
\begin{equation}\label{pre}\Fc fg\Fc\leq (\Fc f\Fc)(\Fc g\Fc).
\end{equation}
 The set $\Fc \mathsf M_n(L)\Fc$
of \Fc-stable elements of $\mathsf M_n(L)$ is a subsemigroup  with identity $\Fc$. Moreover, an element
$f\in \mathsf M_n(L)$ is \Fc-stable if and only if its stable pairs
belong to $L\Fc\times L\Fc$.
\end{Prop}
\begin{proof}
Proposition~\ref{actionfacts} implies $f\mapsto \Fc f\Fc$ is
a closure operator.  Equation \eqref{pre} is an immediate consequence of the inequality $1_L\leq \Fc$.
It is obvious that $\Fc \mathsf M_n(L)\Fc$ is a subsemigroup with identity $\Fc$.  Since as a binary relation, $\Fc = 1_{L\Fc}$, the final statement is clear.
\end{proof}

Recall that if $P$ is a partially ordered set, a subset $K\subseteq P$ is
called a \emph{filter} if $k\in K$ and $p\geq k$ implies $p\in K$.

\begin{Cor}\label{filterprop}
The subset $\Fc \fmon n L\Fc$ is a filter in $\fmon n L$.  In
particular, if $f\in \Fc \fmon n L\Fc$ is $n$-loopable, then
$f^{\omega+\ast}\in \Fc \fmon n L\Fc$.
\end{Cor}
\begin{proof}
Suppose that $f\in \fmon n L$ is $\Fc$-stable and $g\geq f$.  Then
Propositions~\ref{closure} and~\ref{FMF} show \[\mathop{\mathrm{Im}}
g\subseteq \mathop{\mathrm{Im}} f\subseteq L\Fc\times L\Fc.\]  Another
application of Proposition~\ref{FMF} lets us deduce that $g$ is
$\Fc$-stable.

The last statement follows since $f^{\omega}$ is $\Fc$-stable and
$f^{\omega}\leq f^{\omega+\ast}$ by Proposition~\ref{actionfacts}.
\end{proof}

\begin{Def}[$\Fc$-stable transformation monoid]\label{stabletm}
There is a monoid action of $\Fc \fmon n L\Fc$ on $L\Fc$ by total
functions defined by $\ell\cdot f = \ell\overrightarrow f$ for
$\ell\in L\Fc$, $f\in \Fc\fmon n L\Fc$.\index{$\Fc$-stable transformation monoid}
\end{Def}

\begin{Prop}
The action in Definition~\ref{stabletm} is well defined.
\end{Prop}
\begin{proof}
Let $\ell\in L\Fc$ and $f\in \Fc \fmon n L\Fc$.  Then by the last
statement of Proposition~\ref{FMF}, $\ell\overrightarrow f$ is
$\Fc$-closed. It now follows from Propositions~\ref{nobackflow} and~\ref{Fnobackflow} that the action is a well defined semigroup
action.  Clearly $\Fc$ acts as the identity on $L\Fc$ since $\overrightarrow{\Fc}=\Fc$ by Remark~\ref{oneastwonoflow}.
\end{proof}

The action of $\Fc \mathsf M_n(L)\Fc$ is not faithful.  In fact, we have:

\begin{Prop}
Suppose that $g\in \Fc \fmon n L\Fc$ is $n$-loopable. Then the equality
$\ell\cdot g^{\omst} = \ell\cdot g^{\omst}g^{\omega}$ holds for any $\ell\in L\Fc$.
\end{Prop}
\begin{proof}
Observe that $\ell\cdot g^{\omst}g^{\omega} =
\ell\cdot g^{\omst}g^{\omega}g$ by
Lemma~\ref{rhodestrickiness}.  So if $\ell_1=\ell\cdot
g^{\omst}g^{\omega}$, then $\ell_1\cdot g = \ell_1$.  Therefore, we have
\[\ell\plstar{\ov{g^{\omega}}}{\ov g}\ell_1,\] that is,
$\ell\mathrel{\ov {g^{\omst}}}\ell_1$.  We conclude that $\ell\cdot
g^{\omst}\leq \ell_1$ by definition of forward-flow.  But
$g^{\omst}g^{\omega}\leq g^{\omst}$ by Lemma~\ref{gomega+*isidempotent}, so $\ell_1\leq \ell\cdot
g^{\omst}$.  This completes the proof.
\end{proof}

\section{The Lower Bound}
This section constructs our lower bound  for complexity.
More precisely, given a finite $X$-generated group mapping monoid
$M$ with distinguished $\R$-class $R$, we shall effectively
construct a collection of $\pv C_n$-inevitable elements of
$\mathsf{SP}(M,X)$, for $n\geq 0$. Then the results of Section~\ref{s:present}
show that a necessary condition for $M$ to have complexity $n+1$ is
that no element of the form $(Y,P)$ of this collection have a block $B$ of $P$
containing distinct $\H$-equivalent elements of $R$, cf.~Corollary~\ref{presentation-flowform-inev}.

\subsection{The evaluation monoid}
We continue to denote $\sp$ by $L$ and fix $n\geq 0$.  The evaluation transformation semigroup will be the
combinatorial object that encodes the $\pv C_n$-inevitable elements of
$L$ as well as an action of a certain submonoid of $\Fc \fmon n
L\Fc$ on these elements.

 First we need the following notion of a well-formed
formula.

\begin{Def}[Well-formed formulae]
Let $X$ be an alphabet.  We define a well-formed formula inductively as follows.  The empty string $\varepsilon$ is a well-formed formula.  Each letter $x\in X$ is a well-formed formula.  If $\tau,\sigma$ are well-formed formulae, then so is $\tau \sigma$.  If $\tau$ is a well-formed formula  that is not a proper power, then $\tau^{\omst}$ is also a well-formed formula.  The set of well-formed formulae is denoted $\Omega(X)$.  Well-formed formulae will be denoted by Greek letters.  As a convention, if $\tau=\sigma^n$ where $\sigma$ is not a proper power, then we set $\tau^{\omst}=\sigma^{\omst}$; in other words, we extract roots before applying the unary operation.\index{well-formed formula}
\end{Def}

We want to interpret well-formed formulae in $\Fc \mathsf M_n(L)\Fc$.
\begin{Def}[Standard Interpretation]
Define recursively a partial function \mbox{$\Upsilon\colon \Omega(X)\to \Fc \mathsf M_n(L)\Fc$} as follows.  Set $\varepsilon\Upsilon =\Fc$ and $x\Upsilon =\Fc x\Fc$ for $x\in X$.  If $\Upsilon$ is already defined on $\tau,\sigma\in \Omega(X)$, set $(\tau\sigma)\Upsilon = \tau\Upsilon \sigma\Upsilon$.  If $\tau\in \Omega(X)$ is not a proper power and $\tau\Upsilon$ is defined and $n$-loopable, set $\tau^{\omst}\Upsilon = (\tau\Upsilon)^{\omst}$. We normally omit $\Upsilon$ and assume that
$\tau\in \Omega(X)$ is being evaluated in $\Fc\mathsf M_n(L)\Fc$ according to the
standard interpretation.   However, when there is danger of confusion
with free flow, we use $\Upsilon$.\index{standard interpretation}
\end{Def}

We can now define the set of $\Fc$-states.
\begin{Def}[\Fc-states]\label{definefstates}
The collection of \emph{$\Fc$-states} is by
definition the smallest subset $\st n L$ of $L$ such that:
\begin{enumerate}
\item {\rm (Points)} $(r,r)\in \st n L$, for all $r\in R$;
\item {\rm (Forward-flow)} $\ell \overrightarrow{\tau}\in \st n L$ for all $\ell\in \st n L$, $\tau \in \Om (X)$;
\item {\rm (Order ideal)} If $\ell\leq \ell'\in \st n L$, then $\ell\Fc\in \st n L$.
\end{enumerate}\index{$\Fc$-states}
\end{Def}

\begin{Rmk}\label{downto2points}
If there exists $(Y,P)\in  \st n L$ and $r\neq s\in Y$ such that $r\flow P s$, then $(\{r,s\},\{\{r,s\}\})\in \st n L$ by Axiom (Order Ideal) since $(\{r,s\},\{\{r,s\}\})$ is $\Fc$-stable by Proposition~\ref{pointsareclosed}.
\end{Rmk}

First we prove that $\Fc$-states are $\Fc$-stable.

\begin{Prop}
The elements of $\st n L$ are $\Fc$-stable.
\end{Prop}
\begin{proof}
We show $L\Fc$ satisfies the axioms of Definition~\ref{definefstates}.  Axiom (Points) holds by Proposition~\ref{pointsareclosed}.   Since
the front of any stable pair of an element of $\Fc\mathsf M_n(L)\Fc$
is $\Fc$-stable by Proposition~\ref{FMF}, Axiom (Forward-flow) holds
for $L\Fc$.  Axiom (Order ideal) trivially holds for $L\Fc$.
This establishes $\st n L\subseteq L\Fc$.
\end{proof}

We do not have any axiom about back-flow since
if $f\in L_1(M)$ is $\Fc$-stable and $\ell$ is an $\Fc$-state, then
$(\ell,B)f = (\ell,\ell\overrightarrow f)$, so there is no
back-flow.

\begin{Def}[Evaluation transformation monoid]
The action, from Definition~\ref{stabletm}, of $\Fc\fmon n L\Fc$ on
$L\Fc$  restricts to an action of $\Omega(X)\Upsilon$  on $\st n L$ by
Axiom (Forward-flow).  The associated faithful transformation monoid is
denoted \[\Ec n L = (\st n L, M(\Ec n L))\] and called the
\emph{evaluation transformation monoid}.  We term $M(\Ec n L)$
the \emph{evaluation monoid}.\index{evaluation transformation monoid}
\end{Def}

\subsection{Action on sets}
In this section, we try to understand how $\Omega(X)$ acts on set-partitions in the set coordinate.  For this reason, will interpret elements of $\Omega(X)$ as elements of $\mathscr C(\mathsf S(M,X)^2)$, as well.

\begin{Def}[Interpretation on sets]
Define $\Psi\colon \Omega(X)\to \mathscr C(\mathsf S(M,X)^2)$ as follows.
Set $\varepsilon\Psi$ to be the identity of $\mathscr C(\mathsf S(M,X)^2)$.  If $x\in X$, then $x\Psi$ is free flow along $x$.  If $\sigma,\tau\in \Omega(x)$, then $(\sigma\tau)\Psi = \sigma\Psi \tau\Psi$.  If $\sigma$ is not a proper power, then $\sigma^{\omst}\Psi = \sigma\Psi^{\omst}$.    Again, we drop $\Psi$ from the notation when no confusion can arise.\index{interpretation on sets}
\end{Def}

Our aim is to establish a compatibility between the interpretation of $\Omega(X)$ in the set flow lattice and the set-partition flow lattice.

\begin{Prop}\label{projecttoset}
Let $\tau\in \Omega(X)$ be in $\dom \Upsilon$.  Then $(Y,P)\xrightarrow{\ov {\tau}} (Y',P')$ implies $Y\xrightarrow{\ov {\tau}} Y'$.
\end{Prop}
\begin{proof}
We go by induction on the recursive construction of well-formed formulae.  If $(Y,P)\xrightarrow{\ov {\Fc}} (Y',P')$, then $Y=Y'$ and so $Y\xrightarrow{\ov{\varepsilon}} Y=Y'$.  If $x\in X$ and $(Y,P)\xrightarrow{\ov{\Fc x\Fc}} (Y',P')$, then $(Y,P)\xrightarrow{\ov{x}} (Y',P')$ and so $Yx\subseteq Y'$. But this says exactly that $Y\xrightarrow{\ov{x}}Y'$.

Assume that the desired implication holds for $\sigma,\tau\in \Omega(X)$ and suppose $(Y,P)\xrightarrow{\ov {\sigma\tau}} (Y',P')$.  Then we can find $(Y'',P'')$ so that \[(Y,P)\xrightarrow{\ov {\sigma}} (Y'',P'')\xrightarrow{\ov {\tau}} (Y',P').\]  Thus by hypothesis we have \[Y\xrightarrow{\ov {\sigma}} Y''\xrightarrow{\ov {\tau}} Y'\] and so $Y\xrightarrow{\ov{\sigma\tau}}Y'$, as required.  Finally, suppose that $\sigma$ satisfies the conclusion of the proposition, $\sigma\Upsilon$ is $n$-loopable and $\sigma$ is not a proper power.  Choose $m$ so that $f^m=f^{\omega}$ for all $f\in \mathsf M_n(\sp)$ and all $f\in \mathscr C(\mathsf S(M,X)^2)$.  Suppose that $(Y,P)\xrightarrow{\ov {\sigma^{\omst}}}(Y',P')$.  Then \[(Y,P)\plstar{\ov{\sigma^{m}}}{\ov {\sigma}}(Y',P')\] and so by assumption on $\sigma$ and the case of products handled above it follows \[Y\plstar{\ov{\sigma^{m}}}{\ov {\sigma}}Y'\] and hence $Y\xrightarrow{\ov {\sigma^{\omst}}}Y'$.  This completes the proof.
\end{proof}

As a corollary, we see that sets only flow forward: there is no back-flow.

\begin{Cor}\label{nosetbackflow}
If $Y\in \mathsf S(M,X)$ and $\tau \in \Omega(X)$, then $Y\overleftarrow{\tau} =Y$.
\end{Cor}
\begin{proof}
Note that $\Upsilon$ is total for $n=0$.
Since $\tau\Upsilon\in \mathcal F_0\mathsf M_0(\sp)\mathcal F_0$ and $(Y,\{Y\})$ is $\mathcal F_0$-stable by Proposition~\ref{pointsareclosed}, it follows  $(Y,\{Y\})\xrightarrow{\ov {\tau}} (Y,\{Y\})\overrightarrow{\tau}$.  Thus $Y\in \dom \tau$ by Proposition~\ref{projecttoset} and hence $Y\overleftarrow{\tau} =Y$.
\end{proof}

Our next goal is to prove that certain set flows yield set-partition flows.

\begin{Prop}\label{dumbpartition}
Let $\tau\in \Omega(X)$.  Then, for $Y,Y'\in \mathsf S(M,X)$, we have $Y\xrightarrow{\ov {\tau}}Y'$ if and only if $(Y,\{Y\})\xrightarrow{\ov {\tau}}(Y',\{Y'\})$ for $\tau\in \dom \Upsilon$.
\end{Prop}
\begin{proof}
The implication from right to left is a consequence of Proposition~\ref{projecttoset} so we handle the forward implication.  Again, we proceed by induction on the recursive definition of well-formed formulae.   If $Y\xrightarrow{\ov {\varepsilon}}Y'$, then $Y=Y'$.  Since $(Y,\{Y\})$ is $\Fc$-stable, by Proposition~\ref{pointsareclosed}, it follows  $(Y,\{Y\})\xrightarrow{\ov {\Fc}}(Y,\{Y\})$.  Next, suppose $x\in X$ and that $Y\xrightarrow{\ov x}Y'$.  Then $Yx\subseteq Y'$.  Since any partial function $Y/\{Y\}\to Y'/\{Y'\}$ is injective and $(Y,\{Y\}), (Y',\{Y'\})$ are $\Fc$-stable by Proposition~\ref{pointsareclosed}, we conclude \[(Y,\{Y\})\xrightarrow{\ov {\Fc x\Fc}}(Y',\{Y'\}).\]

Assume the proposition holds for $\sigma,\tau\in \Omega(X)$ and suppose $Y\xrightarrow{\ov{\sigma\tau}}Y'$. Then there exists $Y''$ so that \[Y\xrightarrow{\ov {\sigma}} Y''\xrightarrow{\ov {\tau}} Y'.\] By assumption, we have \[(Y,\{Y\})\xrightarrow{\ov {\sigma}} (Y'',\{Y''\})\xrightarrow{\ov {\tau}} (Y',\{Y'\})\] and so $(Y,\{Y\})\xrightarrow{\ov{\sigma\tau}}(Y',\{Y'\})$.

Finally, suppose that the desired conclusion holds for $\sigma\in \Omega(X)$ where $\sigma$ is not a proper power and $\sigma\Upsilon$ is $n$-loopable. Choose $m$ so that $f^m=f^{\omega}$ for all $f\in \mathsf M_n(\sp)$ and all $f\in \mathscr C(\mathsf S(M,X)^2)$.  Then the proposition also holds for $\sigma^m$.  Therefore, if $Y\xrightarrow{\ov {\sigma^{\omst}}}Y'$, then \[Y\plstar{\ov{\sigma^{m}}}{\ov {\sigma}}Y'\] and hence \[(Y,\{Y\})\plstar{\ov{\sigma^{m}}}{\ov {\sigma}}(Y',\{Y'\})\] by the induction hypothesis.  We conclude that $(Y,\{Y\})\xrightarrow{\ov {\sigma^{\omst}}}(Y',\{Y'\})$, as required.
\end{proof}

As a corollary, we can determine what happens to the sets when applying elements of $\Omega(X)$ to elements of $\sp$.
\begin{Cor}\label{setsstayput}
Suppose $\tau\in \Omega(X)$ belongs to $\dom \Upsilon$ and that $Y\xrightarrow{\tau} Y' = (Z,Z')$.  Then, for any partitions $P,P'$ on $Y,Y'$, respectively, we have \[(Y,P)\xrightarrow{\tau}(Y',P') = ((Z,Q),(Z',Q'))\] for some partitions $Q,Q'$ on $Z,Z'$, respectively.
\end{Cor}
\begin{proof}
Suppose that $(Y,P)\xrightarrow{\tau}(Y',P') = ((W,S),(W',S'))$.  Then $W\xrightarrow{\ov {\tau}}W'$ by Proposition~\ref{projecttoset}.  Since $Y\leq W$ and $Y'\leq W'$, it follows that \[(Z,Z') = Y\xrightarrow{\tau} Y'\leq W\xrightarrow{\tau}W' = (W,W').\]
On the other hand, $(Z,\{Z\})\xrightarrow{\ov {\tau}}(Z',\{Z'\})$ by Proposition~\ref{dumbpartition}. From the inequalities $(Y,P)\leq (Z,\{Z\})$ and $(Y',P')\leq (Z',\{Z'\})$, we must have $(W,S)\leq (Z,\{Z\})$ and $(W',S')\leq (Z',\{Z'\})$.  Thus $W=Z$ and $W'=Z'$, as required.
\end{proof}

Note that we can define an action of $\Omega(X)$ on subsets of $R$ by setting $Y\cdot \tau = Y\overrightarrow{\tau}$.  This is an action by Corollary~\ref{nosetbackflow} and Proposition~\ref{nobackflow}.  As a consequence of Corollary~\ref{setsstayput} we obtain the following result.

\begin{Thm}\label{secactionok}
Let $(Y,P)$ be $\Fc$-stable and suppose that $\tau\in \Omega(X)$ is in $\dom \Upsilon$.  Then $(Y,P)\cdot \tau = (Y\cdot \tau, P')$ for some partition $P'$ on $Y\cdot \tau$.
\end{Thm}
\begin{proof}
Since $Y\xrightarrow{\tau} B = (Y\overleftarrow{\tau},Y\overrightarrow{\tau})$, Corollary~\ref{setsstayput} guarantees that \[(Y,P)\xrightarrow{\tau} B = ((Y\overleftarrow{\tau},Q),(Y\overrightarrow{\tau},P'))\] for some partitions $Q$ and $P'$ (actually $Y=Y\overleftarrow{\tau}$ and $Q=P$ since there is no back-flow on $\Fc$-stable set-partitions).
\end{proof}

We would like to make a conjecture on what $Y\cdot \tau$ is for $\tau\in \Omega(X)$.

\begin{Def}[Interpretation in $P(M)$]
Define a map $\Lambda\colon \Omega(X)\to P(M)$ recursively as follows.  Set $\varepsilon \Lambda = \{I\}$ where $I$ is the identity of $M$.  Put $x\Lambda = \{[x]_M\}$.  If $\Lambda$ is defined on $\sigma,\tau$, then set $(\sigma\tau)\Lambda = \sigma\Lambda \tau\Lambda$.  If $\sigma$ is not a proper power, put $\sigma^{\omst}\Lambda = \bigcup_{k\geq 0}(\sigma\Lambda)^{\omega}(\sigma\Lambda)^k$.\index{interpretation in $P(M)$}
\end{Def}

\begin{Conjecture}
If $\tau\in \Omega(X)$ and $Y\subseteq R$, then $Y\cdot \tau = Y(\tau\Lambda)$.
\end{Conjecture}

We do know exactly what happens for strings; things are more complicated for higher rank elements of $\Omega(X)$.

\begin{Prop}\label{stringactionissupereasy}
Let $w\in X^*\subseteq \Omega(X)$ and suppose that $(Y,P)$ is $\Fc$-stable.  Let $P=\{B_1,\ldots,B_r\}$ where $B_1,\ldots,B_k$ are the blocks of $P$ with $B_iw\neq \emptyset$. Then $B_1w,\ldots, B_kw$ are disjoint and the equality \[(Y,P)\cdot w\Upsilon=(Yw,\{B_1w,\ldots, B_kw\}) = (Y,P)\overrightarrow{w}\] holds.
\end{Prop}
\begin{proof}
We proceed by induction on $|w|$.  If $|w|=0$, then there is nothing to prove.  Suppose it is true for $u$ and let $w=ux$ with $x\in X$.  Suppose that $B_1,\ldots, B_m$ are the blocks with $B_iu\neq\emptyset$.  Then $k\leq m$. By induction, \[(Y,P)\cdot u\Upsilon=(Yu,\{B_1u,\ldots, B_mu\})=(Z,Q).\]   Since $(Y,P)\overrightarrow{u\Upsilon}=(Z,Q)$ is $\Fc$-stable, Proposition~\ref{Fnobackflow} yields $(Z,Q)\overleftarrow{x} =(Z,Q)$.  Consequently, Proposition~\ref{stringsareasy} implies that \[(Z,Q)\overrightarrow{x} = (Yw,\{B_1w,\ldots, B_kw\})=(W,S).\]  As $x\leq \Fc x\Fc$, it follows that if $(W',S')=(Z,Q)\cdot x\Upsilon$, then we have $(W,S)\leq (W',S')$. Now $W'=Zx=Yw$ by Theorem~\ref{secactionok} and so $(W',S')=(Zx,S')$.  But if $(W,S)<(W,S')$, then two blocks $B_iux,B_jux$ with $i\neq j$ must be contained in a single block of $S'$, contradicting $(Z,Q)\xrightarrow{\ov x}(W,S')$.  Thus $(W,S)=(W,S')$, as required.
\end{proof}

As $N=\langle \Fc x\Fc\mid x\in X\rangle$ is $X$-generated and acts on $\st n L$, it is natural to try and compute complexity using $\mathscr A=(\st n L,X)$ where the transitions come via the action of $N$ on $\st n L$.  One could then define a complete flow $F$ on $\mathscr A$ by $(Y,P)F = (Y,P)$.  If $N$ had complexity at most $n$, then this would prove that $\st n L$ contains all maximal $\pv C_n$-inevitable elements.  Unfortunately, $M\cong N$ and so we do not know the complexity of $N$.

\begin{Prop}\label{Mdivides}
The submonoid $\langle \Fc x\Fc\mid x\in X\rangle$ of $M(\Ec n L)$ is isomorphic to $M$.
\end{Prop}
\begin{proof}
If $r\in R$, then $(r,r)\in \st n L$ and Proposition~\ref{stringactionissupereasy} readily implies, for $w\in X^*$, that $(r,r)\cdot w\Upsilon = (rw,rw)$ (where as usual $rw=\emptyset$ if it is not defined).  Since $M$ acts faithfully on $R$, it follows that $v\Upsilon = w\Upsilon$ implies $[v]_M=[w]_M$ for $v,w\in X^*$.  Conversely, Proposition~\ref{stringactionissupereasy} immediately yields that $[v]_M=[w]_M$ implies $v\Upsilon = w\Upsilon$ for $v,w\in X^*$. This completes the proof.
\end{proof}

\subsection{The lower bound theorem}
We are now ready to prove the lower bound theorem for complexity.    Once again $M$ is a fixed group mapping monoid generated by $X$ and we set $L=\sp$.

\begin{Thm}[The Inevitability Theorem]\label{aperiodicinev}
Each element of $\st n L$ is $\pv C_n$-inevitable.
\end{Thm}
\begin{proof}
Let $\mathscr I$ be the set of $\pv C_n$-inevitable elements of $L$.
We show that it satisfies the axioms of Definition~\ref{definefstates}. This will show that each element of $\st n L$
is $\pv C_n$-inevitable.

As was observed earlier, points are \pv
V-inevitable for any pseudovariety \pv V.  Suppose that $\ell\in \mathscr I$ and $f\in \mathsf M_n(L)$. We show that $\ell\overrightarrow f\in \mathscr I$. This will imply Axiom (Forward-flow). If $\ell\overrightarrow{f}=B$, there is nothing to prove. So assume henceforth that $\ell\overrightarrow f\neq B$. Let
$\Ac=(Q,X)$ be any partial automaton with transition
monoid in $\pv C_n$ and suppose $F$ is a complete flow on $\Ac$. By Corollary~\ref{usablevalues}, there is an
element $t\in f\upsilon_{\pv C_n}$.  Since $\ell$ is $\pv
C_n$-inevitable, there is a state $q\in Q$ such that $\ell\leq qF$. By
the definition of a value, $qF\xrightarrow {\ov f} qtF$.  Since
$(\ell,B)\leq (qF,qtF)$, it follows that \[\ell\xrightarrow f B =
(\ell\overleftarrow f,\ell\overrightarrow f)\leq (qF,qtF);\] in particular, $qt\neq \square$ as $\ell\overrightarrow{f}\neq B$.  Since
$\Ac$ and $F$ were arbitrary, we deduce that $\ell\overrightarrow
f\in \mathscr I.$

For Axiom (Order ideal), suppose that $\ell'\in \mathscr I$ and $\ell\leq \ell'$.  Since the $\pv C_n$-inevitable elements of $L$ form an order ideal, it follows that $\ell\in \mathscr I$.  So let $F$ be a complete flow on a partial automaton $\mathscr A$ with transition monoid in $\pv C_n$.  Then $\ell\leq qF$ for some state $q$.  But $qF$ is $\Fc$-stable by Lemma~\ref{aperiodicvaluesareFstable}.  Therefore, $\ell\Fc\leq qF$ establishing that $\ell\Fc\in \mathscr I$.
\end{proof}

This leads to our lower bound for complexity, which is the main result of this paper.

\begin{Thm}[Lower Bound Theorem for Complexity]\label{thelowerboundtheorem}
Suppose $M$ is a finite $X$-generated group mapping monoid with
distinguished $\R$-class $R$.  If there exists $r\neq s\in R$ such that $r\H s$ and \[(\{r,s\},\{\{r,s\}\})\in \st n {\sp},\] then $M$ has complexity at least $n+2$.
\end{Thm}
\begin{proof}
This follows directly from Theorem~\ref{aperiodicinev}, Corollary~\ref{presentation-flowform-inev} and Remark~\ref{downto2points}.
\end{proof}

We remark that the proof of Theorem~\ref{aperiodicinev} would seem
to indicate that in Axiom (Forward-flow) we should allow any element
of $\Fc\mathsf M_n(L) \Fc$ to be used.  But it follows easily from
the definition of $\mathsf M_n(L)$, from the fact that $\Fc$ is a
closure operator and from \eqref{pre} that each element of $\mathsf
M_n(L)$ is below an element determined by the standard
interpretation of a well-formed formula.  More precisely, the set of
elements of $\mathscr C(L^2)$ that are below the standard
interpretation of a well-formed formula satisfies the axioms of
Definition~\ref{flowmonoid}.  Hence we would obtain no new maximal
$\pv C_n$-inevitable elements by allowing these other elements.  But only maximal elements are needed to
obtain the lower bound.

\section{Example: The Tall Fork}
The Tall Fork $F$ is a semigroup that was constructed by the second author in order to show that the Type I-Type II lower bound of~\cite{lowerbounds2} is not tight.  A description of $F$ can be found at the beginning of~\cite[Section 4.14]{qtheor} and we shall follow the notation therein religiously.  We also adjoin an identity $I$ to $F$ to make it a monoid $F^I$.  Of course, $F$ and $F^I$ have the same complexity.  The complexity of $F$ is at most $2$ by the Depth Decomposition Theorem~\cite{TilsonXI}.  We use our lower bound to show the complexity of $F^I$ is at least $2$.  To do so, it will be convenient to use the following form of the ``Tie-Your-Shoes'' Lemma~\cite[Lemma 4.14.29]{qtheor}.

\begin{Lemma}[Tie-your-shoes]\label{tieyourshoes}
Suppose  $R$ is the distinguished $\R$-class of an $X$-generated group mapping monoid $M$.  Assume the $\J$-class $J$ of $R$ has Rees matrix coordinatization $J^0\cong\MM^0(G,A,B,C)$. Let $a\in A$ correspond to the $\R$-class $R$.  Suppose that
$b_1Ca_0\neq 0\neq b_2Ca_0$.  Then if $(Y,P)\in \sp$ is $\Fc$-closed and
\[x=(a,g(b_1Ca_0)\inv,b_1),\ y=(a,g(b_2Ca_0)\inv,b_2)\in Y\] some $g\in G$, then $x\flow P y$.
\end{Lemma}
\begin{proof}
Suppose that $P=\{B_1,\ldots, B_r\}$ and that $x\in B_i$ and $y\in B_j$ with $i\neq j$.
Let $w\in X^*$ be a string mapping to the element $z=(a_0,1,b_1)$.  Note that $xz=(a,g,b_1)=yz$.  Proposition~\ref{stringactionissupereasy} then implies that $B_iz=B_iw$ and $B_jz=B_jw$ are disjoint, contradicting $(a,g,b_1)\in B_iz\cap B_jz$.  Thus $i=j$, i.e., $x\flow P y$.
\end{proof}

We use all of $F$ as a generating set $X$ for $F^I$. Remember that elements of $\Omega(X)$ act via $\Upsilon$, which we omit from the notation.  Set $L=\mathsf{SP}(F^I,X)$ and put $r=(a_0,1,0')$.
Let us begin by observing that $(r,r)\in \st 0 L$.  Let $(Y_1,P_1) = (r,r)\cdot k^{\omst}$.  It is easy to see that \[r\cdot k^{\omst} = r\bigcup K = \{(a_0,1,0'),(a_0,1,2')\}.\]  Hence Theorem~\ref{secactionok}, yields $Y_1=\{(a_0,1,0'),(a_0,1,2')\}$.  The Tie-your-shoes Lemma then implies that $P=\{Y_1\}$.  Thus $(Y_1,\{Y_1\})\in \st 0 L$.  Let $(Y_2,P_2)= (Y_1,\{Y_1\})\cdot t$.  Proposition~\ref{stringactionissupereasy} then implies \[Y_2=\{(a_0,-1,0),(a_0,1,2)\}\] and $P_2=\{Y_2\}$.  Consider now $(Y_3,P_3)=(Y_2,\{Y_2\})\cdot h^{\omst}$.  One easily verifies that $Y_2\cdot h^{\omst} = Y_2\bigcup H = \{a_0\}\times \{\pm 1\}\times \{0,1,2,3\}$ and so Theorem~\ref{secactionok} implies that $Y_3= \{a_0\}\times \{\pm 1\}\times \{0,1,2,3\}$.  Since \[(Y_2,\{Y_2\})\xrightarrow{\ov{h^{\omst}}}(Y_3,P_3)\] is stable, we have \[(Y_2,\{Y_2\})\xrightarrow{\ov{h^{\omega}}}(Y_3,P_3)\] and so $Y_2$ is contained in a single block of $P_3$.  Repeated application of the Tie-your-shoes Lemma then establishes $P_3=\{Y_3\}$.  Thus $(Y_3,\{Y_3\})\in \st 0 L$ and hence \[\left(\{(a_0,1,0),(a_0,-1,0)\},\{\{(a_0,1,0),(a_0,-1,0)\}\}\right)\in \st 0 L.\]  Thus $F^I$ has complexity at least $2$ by Theorem~\ref{thelowerboundtheorem}.
\printindex
\bibliographystyle{abbrv}
\bibliography{standard2}
\end{document}